\documentclass[11pt]{amsart}

\usepackage[margin=2cm]{geometry}

\usepackage{latexsym, amssymb, amsfonts, amsmath, amsthm, graphics, graphicx, subfigure, bbm, stmaryrd,xcolor}
\usepackage{epstopdf}
\usepackage[pdftex,colorlinks,linkcolor=blue,menucolor=red,backref=false,bookmarks=true,citecolor=red]{hyperref}

\usepackage{palatino}

\newtheorem{theorem}{Theorem}

\newtheorem{corollary}[theorem]{Corollary}

\newtheorem{proposition}[theorem]{Proposition}

\newcommand{\field}[1]{\mathbb{#1}}
\newcommand{\R}{\field{R}}

\newcommand{\E}{\field{E}}
\newcommand{\C}{\field{C}}

\newcommand{\p}{\field{P}}

\newcommand{\Tr}{\mathrm{Tr}}

\newcommand{\Var}{\mathrm{Var}}

\newcommand{\met}[2][ccccccccccccccccccccccccccccccccc]{\left[ \begin{array}{#1} #2 \\ \end{array}\right]}

\numberwithin{equation}{section}

\title{\bf Refinements of the One Dimensional Free Poincar\'e Inequality}
\author{Christian Houdr\'e} 
\address{School of Mathematics, Georgia Institute of Technology, 
Atlanta, GA 30332, USA.} 
\email{houdre@math.gatech.edu} 
\author{Ionel Popescu} 
\address{School of Mathematics, Georgia Institute of Technology, 686 Cherry Street, 
Atlanta, GA 30332, USA and  ``Simion Stoilow'' Institute of Mathematics 
of the Romanian Academy, 21 Calea Grivi\c tei, 
Bucharest, ROMANIA.}
\email{ipopescu@math.gatech.edu,  ionel.popescu@imar.ro}

\thanks{CH: Research supported in part by 
a Simons Foundation Fellowship grant \#267336.  Many thanks to 
the Laboratoire MAS of \'Ecole Centrale, Paris, and the LPMA of the Universit\'e Pierre et Marie Curie 
for their hospitality while part of this work was done.  IP:partially supported by 
the Romanian National Authority for Scientific Research, CNCS - UEFISCDI, project 
number PN-II-RU-TE-2011-3-0259  and by a Marie Curie Action Grant PIRG.GA.2009.249200.}

\begin{document}

\date{}

\begin{abstract}
We present two extensions of the one dimensional free 
Poincar\'e inequality similar in spirit to 
two classical refinements.   

\end{abstract}

\maketitle

\section{Introduction}

The classical Poincar\'e inequality for a probability measure $\mu$ on $\R^{d}$  
states that there is a constant $\lambda>0$, such that for any real-valued compactly supported 
smooth function $\phi$ defined on $\R^d$, 
\begin{equation}\label{ei:0}
\lambda\Var_{\mu}(\phi)\le \int |\nabla \phi|^{2}d\mu,
\end{equation}
where $\Var_{\mu}(\phi)=\int \phi^{2}\,d\mu-(\int \phi\,d\mu)^{2}$ is the variance of 
$\phi$ with respect to $\mu$.   Another, well known, interpretation of this inequality is 
to view $\lambda$ as the spectral gap of the generator $L$ of the Dirichlet 
form $\Gamma(\phi,\phi)=\int |\nabla \phi|^{2}\,d\mu$ for which $\mu$ is an invariant 
measure.    

For the standard Gaussian measure $\mu$ in $\R^{d}$, an extension of the classical Poincar\'e  
inequality due to Houdr\'e-Kagan~\cite{MR1308667} 
states that, for any smooth compactly supported 
function $\phi$ on $\R^{d}$ and $n\ge1$, 
\begin{equation}\label{ei:1}
\sum_{k=1}^{2n}\frac{(-1)^{k+1}}{k!}\int |\nabla^{k}\phi|^{2}d\mu
\le \Var_{\mu}(\phi)\le \sum_{k=1}^{2n-1}\frac{(-1)^{k+1}}{k!}\int |\nabla^{k}\phi|^{2}d\mu.   
\end{equation}
This last inequality which can be viewed as a Taylor type expansion 
for $\Var_{\mu}(f)$ was extended to a general Markov operators 
framework by Ledoux in \cite{L7}. 

In a different direction \eqref{ei:0} was also extended 
by Brascamp-Lieb~\cite[Theorem 4.1]{Brasc-Lieb} to 
measures on $\R^{d}$ of the form $\mu(dx)=e^{-V(x)}dx$ 
with $V''(x)$ positive definite at each point $x\in\R^{d}$.  
The extension asserts that 
\begin{equation}\label{ei:2}
\Var_{\mu}(\phi)\le \int \langle (V^{\prime\prime})^{-1}\nabla \phi,\nabla \phi\rangle d\mu, 
\end{equation}
for any compactly supported function $\phi$.  
For one dimensional measures, a further extension of \eqref{ei:0} in 
the spirit of \eqref{ei:1} is also possible (see \cite{i10}).   

With the recent interest in high dimensional phenomena, it is quite 
natural to ask what happens with these functional inequalities in the limit.  
One such setup is to apply the classical inequalities to some standard 
random matrices models and then analyze the limiting object.   
Since large random matrices have deep connections with free probability, 
it is also natural to interpret 
the limiting inequalities as the free counterparts of the classical inequalities.   
This is particularly true in view of the random matrix approach, as developed in \cite{HPU1},  
to the Biane-Voiculescu~\cite{BV} transportation inequality 
and the free Log-Sobolev inequality, which first appeared 
in \cite{MR1851716}, and was subsequently analyzed with random 
matrices in \cite{Biane2}. 

In \cite{i2} and \cite{i5} some of these inequalities are studied as stand alone 
inequalities and analyzed using tools from mass transportation.  In \cite{i5}, a version 
of the Poincar\'e inequality is introduced using random matrix 
heuristics but proved without references  
to random matrix models.  For the standard semicircular 
law $\alpha(dx)=\mathbbm{1}_{[-2,2]}(x)\frac{\sqrt{4-x^{2}}}{2\pi}dx$,  
this Poincar\'e inequality states that for any smooth 
function $\phi$ on $[-2,2]$, 
\begin{equation}\label{ei:3}
\int_{-2}^{2}\int_{-2}^{2}\left( \frac{\phi(x)-\phi(y)}{x-y} 
\right)^{2}\frac{(4-xy)}{4\pi^{2}\sqrt{(4-x^{2})(4-y^{2})}}dx\,dy\le \int (\phi' )^{2}\,d\alpha.
\end{equation}
Note that left-hand side of \eqref{ei:3}, which replaces the classical variance term, is 
essentially the fluctuation quantity of random matrices.   Further, note that 
\eqref{ei:3} has a different flavor  
from its classical counterpart.  For example, in case of the 
standard Gaussian measure, 
\eqref{ei:0} is the expression of the spectral gap of the Ornstein-Uhlenbeck 
operator.   In the free case, and as shown in \cite{i5}, \eqref{ei:3} is equivalent to 
\[
\mathcal{N}\le \mathcal{L}
\]
where $(\mathcal{L}\phi)(x)=-(4-x^{2})\phi^{\prime\prime}(x)+x\phi^\prime(x)$ and $\mathcal{N}$ 
are respectively the Jacobi operator and the counting number operator for the 
orthonormal basis of Chebyshev polynomials $T_{n}(x/2)$ of $L^{2}(\beta)$, where $\beta$ is the 
arcsine measure $\beta(dx)=\mathbbm{1}_{[-2,2]}(dx)\frac{dx}{\pi\sqrt{4-x^{2}}}$.    
Here we interpret the operators as unbounded operators on $L^{2}(\beta)$ 
which is to be contrasted with the classical case where the 
left-hand side is simply a 
projection operator.     

The inequality \eqref{ei:3} can be realized as the limiting case of the classical 
Poincar\'e inequality applied to the distribution of the GUE-ensemble.  
On the other hand, the inequality \eqref{ei:0} is valid 
for measures $\mu(dx)=e^{-V(x)}dx$ on $\R^{d}$ with $V^{\prime\prime}(x)\ge\lambda>0$, 
as easily seen from \eqref{ei:2}.  
Now, let $V:\R\to\R$ be such that $V^{\prime\prime}\ge\lambda>0$, and 
let $\mu_{n}(dX)=\frac{1}{Z_{n}(V)}e^{-n\Tr V(X)}dX$ be the corresponding probability measure on 
$n\times n$ Hermitian matrices.   Let further $\mu_{V}$ be the equilibrium measure, i.e., 
the unique probability measure minimizing the functional 
\begin{equation}\label{ei:7}
E_{V}(\mu)=\int V\,d\mu-\iint \log|x-y|\mu(dx)\mu(dy).
\end{equation}
Since $V$ is convex, the support of $\mu_{V}$ is an interval and, up to 
rescaling, we may assume for simplicity that this support is $[-2,2]$.  
In this setup, applying the Poincar\'e inequality to the measures $\mu_n$ and 
functions of the form 
$\Phi(X)=\Tr(\phi(X))$, with $\phi:\R\to\R$ smooth and compactly supported lead (see \cite{i5} 
for details) to the limiting inequality:   
\begin{equation}\label{ei:5}
 \lambda  \int_{-2}^{2}\int_{-2}^{2}\left( \frac{\phi(x)-\phi(y)}{x-y}  
 \right)^{2}\frac{(4-xy)}{4\pi^{2}\sqrt{(4-x^{2})(4-y^{2})}}dx\,dy\le\int (\phi')^{2}d\mu_{V}.
\end{equation}
 
This inequality was further investigated in \cite{i7} in relation to other 
free functional inequalities such as the transportation, the Log-Sobolev, and the HWI 
ones.   The main tool involved there is the counting number operator 
$\mathcal{N}$ alluded to above and given by  
\[
 (\mathcal{N}\phi)(x) =\int y\phi^\prime(y)\beta(dy)  +x\int \phi^\prime(y)\beta(dy)
       - (4-x^{2})\int \frac{\phi^\prime(x)-\phi^\prime(y)}{x-y} \, \beta(dy).  
\]

A first primary purpose of the present paper is to refine the 
inequality \eqref{ei:3}, which is the free Poincar\'e inequality 
for the semicircular law in the spirit of the classical refinement \eqref{ei:1}.   
The corresponding statement 
is that for any smooth function $\phi$ on $[-2,2]$, and any positive integer $k$, 
\[
\sum_{l=1}^{2k}\!\frac{(-1)^{l-1}}{l}\|\partial^{(l-1)}\phi^\prime\|^{2}_{\alpha^{\otimes l}}
\le\!\int_{-2}^{2}\!\int_{-2}^{2}\!\!\left(\!\frac{\phi(x)
-\phi(y)}{x-y}\!\right)^2\!\!\!\!\frac{(4-xy)}{4\pi^{2}\sqrt{(4-x^{2})(4-y^{2})}}dx\,dy
\le\!\sum_{l=1}^{2k-1}\!\frac{(-1)^{l-1}}{l}\|\partial^{(l-1)}
\phi^\prime\|^{2}_{\alpha^{\otimes l}},
\]
where $\partial$ is the non-commutative derivative introduced in \cite{V5}, 
and where the $\partial^{(l)}$ are its higher versions.   
The above inequality is, in fact, a 
consequence of an exact representation with remainder 
(depending on $\mathcal{M}$, the counting number operator for 
the rescaled Chebyshev polynomials 
of the second kind) for the sandwiched term.   
This is contained in Theorem~\ref{t:10}.  
The proof of this result is based on two main ingredients, 
a first one is the basic relation 
between the operator 
$\mathcal{N}$ and $\mathcal{M}$ which appears 
in Theorem~\ref{t:2} and states that 
\begin{equation}\label{ei:6}
\langle\mathcal{N}\phi,\phi\rangle
=2\langle (\mathcal{M}+I)^{-1}\phi^\prime,\phi^\prime \rangle_{\alpha},
\end{equation}
where the inner product on the left-hand side is the one in $L^{2}(\beta)$, while on 
the right-hand side is the one in $L^{2}(\alpha)$.  
This statement, by itself, is enough to get the 
free Poincar\'e inequality \eqref{ei:3} which follows from that 
$\mathcal{M}$ is a non-negative operator.   

The second ingredient is based on an idea exposed in \cite{i10} which gives a 
refinement of the Brascamp-Lieb inequality \eqref{ei:2} in the spirit of the 
expansion from \eqref{ei:1}.  At its roots there are some commutation relations.   
To wit a bit on this idea, the starting point is the fact that 
\[
\langle (\mathcal{M}+I)^{-1}\phi^\prime,\phi^\prime \rangle_{\alpha}
=\langle \phi^\prime,\phi^\prime\rangle -\langle \mathcal{M}(\mathcal{M}+I)^{-1}\phi^\prime,\phi^\prime \rangle_{\alpha}
\]
and that $\mathcal{M}=\partial^{*}\partial$, where $\partial^{*}$ is the 
adjoint of $\partial$.  This can then be continued with 
\[
\langle \mathcal{M}(\mathcal{M}+I)^{-1}\phi^\prime,\phi^\prime \rangle_{\alpha}
=\langle \partial(\mathcal{M}+I)^{-1}\phi^\prime,\partial\phi^\prime \rangle_{\alpha\otimes\alpha}
=\langle (\mathcal{M}^{(2)}+2I)^{-1}\partial\phi^\prime,\partial \phi^\prime \rangle_{\alpha\otimes \alpha},
\]
where $\mathcal{M}^{(2)}$ is an extension of the operator $\mathcal{M}$ to tensors, in a 
natural way, as $\mathcal{M}^{(2)}(P\otimes Q)=(\mathcal{M}P)\otimes Q+P\otimes (\mathcal{M}Q)$, 
for any polynomials $P$ and $Q$.  Along the way, we also 
used an important commutation relation between $\mathcal{M}$ and the 
derivative operator $\partial$.   
Now, an iteration leads to the the basic expansion 
\[
\langle (\mathcal{M}+I)^{-1}\phi^\prime,\phi^\prime \rangle_{\alpha}
=\langle \phi^\prime,\phi^\prime\rangle 
-\langle (\mathcal{M}^{(2)}+2I)^{-1}\partial\phi^\prime,\partial\phi^\prime \rangle_{\alpha\otimes\alpha}.
\]
This procedure can then be pursued to get more terms as 
detailed in Section~\ref{S:1dim}.

As a second purpose, we wish to extend the free Poincar\'e inequality 
\eqref{ei:5} to a free Brascamp--Lieb inequality similar to \eqref{ei:2} 
in the form (presented in Theorem~\ref{t:11})  
\begin{equation}\label{h:1}
\int_{-2}^{2}\int_{-2}^{2}\left( \frac{\phi(x)-\phi(y)}{x-y}\right)^{2}\frac{(4-xy)dx\,dy}
{4\pi^{2}\sqrt{(4-x^{2})(4-y^{2})}}\le \int \frac{{\phi^\prime}^{2}}{V^{\prime\prime}}d\mu_{V},
\end{equation}
which holds for any smooth function $\phi$ on $[-2,2]$.   The main 
idea in proving \eqref{h:1} is similar to one outlined by 
Helffer in \cite{MR1624506} and consists in writing the left-hand side of \eqref{h:1} 
as $\langle (\mathcal{M}_{V} +V^{\prime\prime})^{-1}\phi,\phi \rangle_{\mu_{V}}$, for 
some operator 
${\mathcal M}_{V}$, presumably unbounded and non-negative definite.   
To see what the candidate for ${\mathcal M}_V$ should be, we use heuristics 
from the classical result applied to random matrices.   
Once this operator is settled, then the proof follows once 
it is shown that 
$\langle (\mathcal{M}_{V} +V^{\prime\prime})^{-1}\phi,\phi \rangle_{\mu_{V}}$ 
does not depend on the potential $V$.  
If this is indeed the case, then choosing our favorite potential, namely, 
$V(x)=x^{2}/2$, 
then the left-hand side of \eqref{h:1} is nothing but \eqref{ei:6}.  
To some extent, at the bottom of this argument is 
the fact that the variance term on the left-hand side 
of \eqref{h:1} is universal, by which we mean 
universality of the fluctuations of random matrices.   

Both extensions provided here are sharp, i.e., we can find non-trivial 
functions for 
which equality is attained.  

The paper is organized as follows.   Section~\ref{s:1} introduces the main notations 
and some preliminary facts.    Section~\ref{s:2} contains the main operators and 
their interrelations which are partially imported from \cite{i7}.   
Section~\ref{s:3} is an intermezzo containing an interpolation equality 
which parallels the classical one.  It is also a motivation for a brief description 
for the free Ornstein-Uhlembeck semigroup seen from a different perspective.  
This, in turn, provides yet another (and simple) 
proof of the free one dimensional Poincar\'e inequality for the semicircular law.  
Section~\ref{S:1dim} gives the main refinement associated with the semicircular law, 
extending the operator $\mathcal{M}$ to tensors and 
unearthing the main commutation relations.  These are then used 
in the expansion of the variance like term, finally leading to the free version of 
Houdr\'e-Kagan \eqref{ei:1}.  Next, Section~\ref{s:heur} is a purely heuristic section which 
motivates the introduction of the main operator associated to the equilibrium measure for 
a potential $V$.   It also serves as a quick recapitulation of Helffer's arguments 
(from \cite{MR1624506}) on obtaining the Brascamp-Lieb result.  
These are, finally, used in Section~\ref{s:7} 
to prove the free version of Brascamp-Lieb \eqref{ei:2}.

\section{Preliminaries}\label{s:1}

\subsection{Random matrices and logarithmic potentials}  

The random matrix ensembles we deal with 
are prescribed by the probability measures on the set $\mathcal{H}_{n}$ of 
$n\times n$ Hermitian matrices determined by a potential $V:\R\to\R$ via
\begin{equation}\label{e0:0}
\p_{V}^{n}(dX)=\frac{1}{Z_{n}(V)}e^{-n\Tr V(X)}dX.
\end{equation}
Here, 
\[
Z_{n}(V)=\int e^{-n\Tr V(X)}dX
\]
is simply the normalizing constant which makes 
$\p_{V}^{n}$ a probability measure.    

It is known, see \cite{Deift1} or \cite{J}, that for any 
such $V$ 
\begin{equation}\label{e0:1}
-\lim_{n\to\infty}\frac{1}{n^{2}}\log Z_{n}(V)= E_{V}=\inf\left\{ E_{V}(\mu):\mu\in\mathcal{P}(\R)\right\}, 
\end{equation}
where 
\[
E_{V}(\mu):=\int Vd\mu-\iint \log|x-y|\mu(dx)\mu(dy),
\]   
and $\mathcal{P}(\R)$ is the set of probability measures on $\R$.  
For $V$ having enough growth at infinity, for instance, if 
\[
\lim_{|x|\to\infty}(V(x)-2\log|x|)=\infty,
\]
this minimization problem is known to have a unique solution $\mu_{V}$ and 
standard references on this are \cite{Deift1,J,ST}.  
The variational characterization of the measure $\mu_{V}$ is 
\begin{equation}\label{e0:4}
\begin{split}
V(x)&\ge 2\int \log|x-y|\mu(dy)+C \quad\text{quasi-everywhere on }\R \text{ and } \\ 
V(x)&= 2\int \log|x-y|\mu(dy)+C 
\quad\text{quasi-everywhere on}\:\: \mathrm{supp}\mu. \\
\end{split}
\end{equation} 
Therefore, taking the derivative in the second line of \eqref{e0:4}, it follows that
on the support of $\mu_{V}$ (assuming the support is a finite union of intervals),  
\begin{equation}\label{e0:5}
V^\prime(x)= p.v.\int\frac{2}{x-y}\mu_{V}(dy), 
\end{equation}
where, as usual, $p.v.$~stands for the Cauchy principal value.   
In this paper we limit ourselves to a smooth $V$ which is also convex, in which case 
the support of the measure 
$\mu_{V}$ is a single interval (see \cite{ST}).  
We can, in fact, weaken the smoothing condition on $V$ and for the main result, 
it suffices for $V$ to be $C^{4}$--regular.   
To shorten the notations, we also denote the principal value of a 
measure $\nu$ by $H\nu$, i.e., $(H\nu)(x):=p.v.\int\frac{2}{x-y}\nu(dy)$.

In addition to \eqref{e0:1}, another important convergence property is that 
for any bounded continuous function $g$ on the real line,
\begin{equation}\label{e0:2}
\int \frac{1}{n}\Tr (g(X))\,\p_{V}^{n}(dX)\xrightarrow[n\to\infty]{} \int g \,d\mu_{V}.
\end{equation}
In fact, something even stronger takes place here, namely, $\frac{1}{n}\Tr (g(X))$ converges 
almost surely to 
$\int g \,d\mu_{V}$, as it can be seen, for instance, from \cite{MR1746976}.  
Above, the $GUE(\frac{1}{\sqrt{n}})$ ensemble corresponds to $V(x)=x^{2}/2$.

Let us now turn to some of the basic operators which play an important 
r$\hat{\text o}$le in the treatment of the 
free Poincar\'e inequality.   There are two important 
measures on $[-2,2]$, the semicircular one 
and the arcsine one, respectively defined by    
\begin{equation}\label{e0:ab}
\alpha(dx)=\frac{\sqrt{4-x^{2}}}{2\pi}dx, \quad \text{ and }\quad \beta(dx)
=\frac{dx}{\pi\sqrt{4-x^{2}}}.  
\end{equation}
Most of the action takes place around the arcsine measure $\beta$ and so we 
use $\langle\cdot,\cdot \rangle$ to denote the inner product 
in $L^{2}(\beta)$, while for any other measure $\mu$, 
$\langle\cdot,\cdot \rangle_{\mu}$ is the inner product in $L^{2}(\mu)$.   
The reason for dealing with the interval $[-2,2]$ is that the 
semicircular law $\alpha$ has mean zero and variance one.  
Another important reason is unfolded in \cite{i6} and \cite{i7} and 
stems from the prominent r$\hat{\text o}$le played by the 
Chebyshev polynomials in analyzing the logarithmic potentials.  
Thus, for convex potentials $V$ the support of the equilibrium measure is a 
single interval.  Thus, by rescaling, namely replacing 
$V(x)$ by $V(c x+b)$ for appropriate $c>0$ and $b$ real, 
the support of $\mu_{V}$ can always be arranged to be $[-2,2]$. 

Another measure which plays a r$\hat{\text o}$le in the sequel is  
\begin{equation}\label{e:omega}
\omega(dx,dy)=\mathbbm{1}_{[-2,2]}(x)\mathbbm{1}_{[-2,2]}(y)
\frac{(4-xy)}{4\pi^{2}\sqrt{(4-x^{2})(4-y^{2})}}dx\,dy. 
\end{equation}

We introduce next the appropriate orthogonal basis associated 
to the measures $\alpha$ and $\beta$.  
These are
\begin{equation}\label{e0:6}
\phi_{n}(x)=T_{n}\left(\frac{x}{2} \right) \quad
    \text{ and } \quad \psi_{n}(x)=U_{n}\left(\frac{x}{2} \right), \quad \text{ for }n\ge0. 
\end{equation}
Here $T_n(x)$ is the $n$th {\em Chebychev polynomials of the first kind} defined via  
$T_n(\cos\theta)=\cos(n\theta)$, while $U_{n}$ is the 
$n$th \emph{Chebyshev polynomials of the second kind} defined  
via  $U_{n}(\cos \theta)=\frac{\sin (n+1)\theta}{\sin\theta}$.   
Adjusting a little the polynomials $T_{n}$ as $\tilde{T}_{0}=T_{0}$ 
and $\tilde{T}_{n}(x)=\sqrt{2}T_{n}(x)$,  it is easily seen 
that $\{\tilde{T}_{n}(x/2)\}_{n\ge0}$ is an orthonormal basis for $L^{2}(\beta)$.   
Similarly, $\{U_{n}(x/2)\}_{n\ge0}$ forms an orthonormal basis for $L^{2}(\alpha)$.   
Other relations between these functions, of later use and, which can be checked 
effortlessly include
\begin{equation}\label{e0:7}
\phi_{n}^\prime=\frac{n}{2}\psi_{n-1}, 
\end{equation}
and
\begin{equation}\label{e0:7b}
-2\psi_{n-1}'(x)+\frac{x}{2}\psi_{n}'(x)=n \psi_{n}(x).
\end{equation}

A further fact, used several times below, is the following relationship:  
\begin{equation}\label{ep:22}
\frac{\psi_{n}(x)-\psi_{n}(y)}{x-y}=\sum_{k=0}^{n-1}\psi_{k}(x)\psi_{n-1-k}(y), 
\end{equation}
which, for instance, can be deduced from the expression for the 
generating function of the Chebyshev polynomials of the second kind 
given by:   
\[
\sum_{n = 0} ^\infty r^{n}U_{n}(x)=\frac{1}{1-2rx+r^{2}} \, , \quad  r\in(-1,1).
\]

\subsection{Random matrices and fluctuations}  

By the study of the fluctuations associated to 
the random matrix models introduced above, we mean the study of the limiting 
behavior of $\mathrm{Tr}(\phi(M))- \E[ \mathrm{Tr}(\phi(M))]$, as $n$ tends to $\infty$, e.g., 
see \cite{J} and \cite{Pastur}.   Assuming that $\mu_{V}$ is supported on $[-2,2]$, 
the variance of this random variable, with respect to $\p_{V}^{n}$ is given in the limit by 
\begin{equation}\label{ep:23}
\int_{-2}^{2}\int_{-2}^{2}\left( \frac{\phi(x)-\phi(y)}{x-y}\right)^{2}
\frac{(4-xy)}{4\pi^{2}\sqrt{(4-x^{2})(4-y^{2})}}dx\,dy
=\iint\left( \frac{\phi(x)-\phi(y)}{x-y}\right)^{2}\omega(dx,dy), 
\end{equation}
a quantity which plays in our context a r$\hat{\text o}$le analogous 
to the one of the variance in the classical setting.

\subsection{Semicircular systems}\label{s:SS}  

Here we summarize a few facts about 
the $R$-transform and introduce the notion of a semicircular system.   

A non-commutative probability space is a pair $(\mathcal{A},\phi)$, where 
$\mathcal{A}$ is a unital $*$-algebra 
and $\phi$ is a trace on it such that $\phi(1)=1$.   
For basic notions of freeness we refer the reader to \cite{MR1217253}.  Nevertheless, 
we mention here the version of $R$-transform in the spirit of \cite{MR1436532}.  
All non-commutative variables $a,b$ considered in this section are 
assumed to be self-adjoint, i.e. $a^{*}=a$ and $b^{*}=b$.  
 
Now, given non-commutative variables $a_{1},a_{2},\dots,a_{n}$ in $\mathcal{A}$, 
the moment generating function of $(a_{1},a_{2},\dots,a_{n})$ is the formal power 
series in non-commuting variables $z_{1},z_{2},\dots,z_{n}$ described by 
\[
M_{a_{1},a_{2},\dots,a_{n}}(z_{1},z_{2},\dots,z_{n})=\sum_{s=1}^{\infty}\sum_{i_{1},i_{2},\dots,i_{s}
=1}^{n}\phi(a_{i_{1}}a_{i_{2}}\dots a_{i_{s}})z_{i_{1}}z_{i_{2}}\dots z_{i_{s}}. 
\] 
The $R$-transform is also a formal power series in non-commuting variables 
$z_{1},z_{2},\dots,z_{n}$ described by 
 \[
 R_{a_{1},a_{2},\dots,a_{n}}(z_{1},z_{2},\dots,z_{n})
 =\sum_{s=1}^{\infty}\sum_{i_{1},i_{2},\dots,i_{s}=1}^{n}
 k_{s}(a_{i_{1}},a_{i_{2}},\dots,a_{i_{s}})z_{i_{1}}z_{i_{2}}\dots z_{i_{s}},
 \]  
where the $k_{s}$ are the free cumulants.  
The moment generating function and the $R$ transforms are related by 
\begin{equation}\label{e:MR}
M=R\boxed{\star}Moeb \text{ and } R=M\boxed{\star}Zeta, 
\end{equation}
where $\boxed{\star}$ is described in \cite{MR1268597} and also 
in \cite{MR1436532} in terms of the lattice of the non-crossing partitions.  
Here
\[
Zeta(z_{1},z_{2},\dots,z_{n})=\sum_{s=1}^{\infty}
\sum_{i_{1},i_{2},\dots,i_{s}=1}^{n}z_{i_{1}}z_{i_{2}}\dots z_{i_{s}}
\]  
and 
\[
Moeb(z_{1},z_{2},\dots,z_{n})
=\sum_{s=1}^{\infty}\sum_{i_{1},i_{2},\dots,i_{s}=1}^{n}(-1)^{s+1}
\frac{(2s-2)!}{(s-1)!s!}z_{i_{1}}z_{i_{2}}
\dots z_{i_{s}}\]
are the Zeta and Moebius functions in $n$ variables associated to the lattice of 
non-crossing partitions (see \cite[Eqs.  3.10 and 3.11]{MR1436532}).  
The only point we need to make here is that $R$ determines $M$,  
and that vice versa $M$ determines $R$.  In particular, if 
$R_{a_{1},a_{2},\dots,a_{n}}(z_{1},z_{2},\dots,z_{n})= 
R_{b_{1},b_{2},\dots,b_{n}}(z_{1},z_{2},\dots,z_{n})$, 
then $M_{a_{1},a_{2},\dots,a_{n}}(z_{1},z_{2},\dots,z_{n}) 
= M_{b_{1},b_{2},\dots,b_{n}}(z_{1},z_{2},\dots,z_{n})$ which means 
that $\phi(a_{i_{1}}a_{i_{2}}\dots a_{i_{s}})=\phi(b_{i_{1}}b_{i_{2}}\dots b_{i_{s}})$, 
or otherwise stated, the mixed moments are the same.  
 
A main property of the $R$-transform is 
that $(a_{1},a_{2},\dots,a_{n})$ and $(b_{1},b_{2},\dots,b_{m})$ are free 
if and only if 
\begin{equation}\label{e:500}
R_{a_{1},a_{2},\dots,a_{n},b_{1},b_{2},\dots,b_{m}}(z_{1},z_{2},\dots,z_{n},z_{1}',z_{2}',\dots,z_{m}')=R_{a_{1},a_{2},\dots,a_{n}}(z_{1},z_{2},\dots,z_{n})
+ R_{b_{1},b_{2},\dots,b_{m}}(z_{1}',z_{2}',\dots,z_{m}'). 
\end{equation}
The second property is that if $a$ is a standard semicircular element 
(i.e. $\phi(a^{k})=\int x^{k}\alpha(dx)$), then 
\begin{equation}\label{e:501}
R_{a}(z)=z^{2}.
\end{equation}

Next, we say that a tuple $(b_{1},b_{2},\dots,b_{n})$ is a 
\emph{standard semicircular system} if the variables are free and each 
of them is a standard semicircular element.  
In particular, in light of \eqref{e:500} and \eqref{e:501}, this means that 
\begin{equation}\label{e:SSS}
R_{b_{1},b_{2},\dots,b_{n}}(z_{1},z_{2},\dots,z_{n})=z_{1}^{2}+z_{2}^{2}+\dots +z_{n}^{2}.  
\end{equation}

Further, we say that a tuple $(a_{1},a_{2},\dots,a_{n})$ of centered variables  
(i.e. $\phi(a_{i})=0$, $1\le i\le n$) is a \emph{semicircular system} if 
\begin{equation}\label{e:SS}
R_{a_{1},a_{2},\dots,a_{n}}(z_{1},z_{2},\dots,z_{n})=\sum_{i,j=1}^{n}c_{ij}z_{i}z_{j}.
\end{equation}
As it turns out, the coefficients $c_{ij}$ are determined by 
$c_{ij}=\phi(a_{i}a_{j})$, hence the matrix $C=\{ c_{ij}\}_{i,j=1}^{n}$ is simply 
the covariance matrix of the tuple.  Moreover, since $\phi$ is a trace, 
$C$ is a real valued symmetric non-negative definite matrix.  
Note that this notion of semicircular system mimics the 
classical notion of a (multidimensional) Gaussian random variable, in that 
the logarithm of the characteristic function is a quadratic function.  
Also, as in the classical case, a semicircular system is completely determined by 
the covariance matrix, by which 
we mean that the mixed moments of $(a_{1},a_{2},\dots,a_{n})$ are determined by 
the covariance matrix and the inversion formula \eqref{e:MR}. 
 
The main results to be used, are contained in 
the following statement.
\begin{proposition}\label{p:1SS}
\begin{enumerate}
\item  Let $(b_{1},b_{2},\dots, b_{m})$ be a semicircular system with 
covariance $C$.  Let $D$ be an $n\times m$ matrix and let $a_{i}=\sum_{j=1}^{m}d_{ij}b_{j}$.  Then $(a_{1},a_{2},\dots,a_{n})$ is a semicircular system with 
covariance matrix $\tilde{C}=DCD^{t}$.   
\item  Let $C$ be an $n\times n$ real symmetric and non-negative definite matrix,  
and let $D=C^{1/2}$.  Then for any standard semicircular system 
$(b_{1},b_{2},\dots, b_{n})$, the tuple $(a_{1},a_{2},\dots,a_{n})$ with 
$a_{i}=\sum_{j=1}^{n}d_{ij}b_{j}$ is a semicircular system with covariance matrix $C$.  
In particular, for any symmetric  non-negative definite matrix $C$, there exists a semicircular system 
with covariance matrix $C$.   

\item If two semicircular systems have the same covariance matrix, then they have the 
same moments.  More precisely, if $(a_{1},a_{2},\dots,a_{n})$  and 
$(b_{1},b_{2},\dots,b_{n})$ are two semicircular systems with the same 
covariance matrix, then $\phi(a_{i_{1}}\dots a_{i_{s}})=\phi(b_{i_{1}}\dots b_{i_{s}})$,  
for any $1\le i_{1},i_{2},\dots,i_{s}\le n$ and any $s\ge1$.   
\end{enumerate}
\end{proposition}

\begin{proof}
\begin{enumerate}
\item  First, by the very definition of the $R$ transform and the linearity of the cumulants, 
it follows that 
\[
\begin{split}
R_{a_{1},a_{2},\dots,a_{n}}(z_{1},z_{2},\dots,z_{n})
&= \sum_{s=1}^{\infty}
\sum_{i_{1},i_{2},\dots,i_{s}=1}^{n}
k_{s}(a_{i_{1}},a_{i_{2}},\dots,a_{i_{s}})z_{i_{1}}z_{i_{2}}\dots z_{i_{s}} \\
&=
\sum_{s=1}^{\infty}\sum_{i_{1},i_{2},\dots,i_{s}=1}^{n} \sum_{j_{1},j_{2},\dots, j_{s}=1}^{m}
d_{i_{1},j_{1}}d_{i_{2},j_{2}}\dots d_{i_{s},j_{s}}
k_{s}(b_{j_{1}},b_{j_{2}},\dots,b_{j_{s}})z_{i_{1}}z_{i_{2}}\dots z_{i_{s}}.
\end{split}
\]
Next, by the definition of $ R_{b_{1},b_{2},\dots,b_{m}}(z_{1},z_{2},\dots,z_{m})$, and 
the quadratic assumption in $z_{1},\dots, z_{m}$, we infer that 
\[
k_{s}(b_{j_{1}},b_{j_{2}},\dots,b_{j_{s}})=0 
\text{ if }s\ne 2 \text{ and } k_{2}(b_{j},b_{l})=c_{jl}. 
\]
In turn, this implies that (denoting by $\tilde{c}_{i,j}$ the entries 
of $\tilde{C}$)
\[
R_{a_{1},a_{2},\dots,a_{n}}(z_{1},z_{2},\dots,z_{n})=
\sum_{i_{1},i_{2}=1}^{n} \sum_{j_{1},j_{2}=1}^{m}d_{i_{1},j_{1}}d_{i_{2},j_{2}}
c_{j_{1},j_{2}}z_{i_{1}}z_{i_{2}} 
= \sum_{i_{1},i_{2}=1}^{n} \tilde{c}_{i_{1},i_{2}}z_{i_{1}}z_{i_{2}}, 
\]
which is precisely what needed to be proved.  
 
\item This follows from the previous item combined with the fact that the covariance 
matrix of a standard semicircular system is the identity matrix.   
 
\item This is the uniqueness of the moment generating function as it follows, for 
example, from \eqref{e:MR}.  \qedhere
\end{enumerate}
\end{proof}

\section{The main operators}\label{s:2}

We are now ready to introduce the main operators of interest.  
For a $C^{2}$ function, $\phi:[-2,2]\to\R$, set
\begin{equation}\label{e0:EN}
\begin{split}
(\mathcal{E}\phi)(x)&=-\int \log|x-y|\phi(y)\beta(dy), \\
(\mathcal{F}\phi)(x)&=-\int \log|x-y|\phi(y)\alpha(dy),  \\
(\mathcal{N}\phi)(x)&=\int y\phi^\prime(y)\beta(dy) 
+x\int \phi^\prime(y)\beta(dy)
- (4-x^{2})\int \frac{\phi^\prime(x)-\phi^\prime(y)}{x-y} \,  \beta(dy), \\ 
(\mathcal{M}\phi)(x)&=2\,p.v.\int \frac{\phi(x)-\phi(y)}{(x-y)^{2}}\alpha(dy)
= \lim_{\epsilon\searrow0} 2\int_{|x-y|\ge\epsilon} \frac{\phi(x)-\phi(y)}{(x-y)^{2}}\alpha(dy).
\end{split} 
\end{equation}

Given a measure $\mu$ on $[-2,2]$, let 
\[
L^{2}_{0}(\mu)=\left\{ f\in L^{2}(\mu): \int f\,d\mu=0\right\}.
\]

Below, is a list of relationships between the operators just defined which is 
mainly imported from \cite[Proposition 1]{i7}.    

\begin{proposition}\label{p:10}
\begin{enumerate}
\item $\mathcal{E}$ maps $C^{2}([-2,2])$ to $C^{2}([-2,2])$ and can be extended to a 
bounded self-adjoint  operator from $L^{2}(\beta)$ into itself.   
\item For any $C^{2}$ function $\phi \in L^{2}_{0}(\beta)$, 
\begin{equation}\label{ep:6}
\begin{split}
\mathcal{E}\mathcal{N}\phi &= \phi, \\
\mathcal{N}\mathcal{E}\phi &=\phi.
\end{split}
\end{equation}
\item $\mathcal{E}\phi_{0}=0$, and for  
$n\ge1$, $\mathcal{E}\phi_{n}=\phi_{n}/n$.  Moreover, for $n\ge0$,  
$\mathcal{N}\phi_{n}=n\phi_{n}$.  In other words, 
$\mathcal{N}$ is the counting number operator for the 
Chebyshev basis $\{\phi_{n}\}_{n\ge0}$ of $L^{2}(\beta)$, and it 
can be canonically extended to a self-adjoint operator on $L^{2}(\beta)$,  
which when restricted to $L^{2}_{0}(\beta)$ has inverse $\mathcal{E}$.  
\item For any $C^{1}$ functions, $\phi$ and $\psi$, on $[-2,2]$,  
\begin{equation}\label{ep:21}
\langle \mathcal{N}\phi,\psi \rangle=2\iint
\frac{(\phi(x)-\phi(y))(\psi(x)-\psi(y))}{(x-y)^{2}} \, \omega(dx,dy),
\end{equation}
\noindent
and in particular, $\langle \mathcal{N}\phi,\psi \rangle=\langle \phi,\mathcal{N}\psi \rangle$.
\item If $V$ is a $C^{3}$ potential on $[-2,2]$ whose equilibrium 
measure $\mu_{V}$ has support $[-2,2]$, then 
\begin{equation}\label{ep:40}
d\mu_{V}=\left(1-\frac{1}{2}\mathcal{N}V\right)\,d\beta.
\end{equation}

\item The operator $\mathcal{M}$ is the counting number operator 
for the basis 
$(\psi_{n})_{n\ge 0}$ of $L^{2}(\alpha)$ and it 
has a natural extension as a self-adjoint operator on $L^{2}(\alpha)$.  
In other words, for any $n\ge0$, $\mathcal{M}\psi_{n}=n\psi_{n}$.  
In addition, for any $C^{1}$ function $\phi$ on $[-2,2]$, 
\begin{equation}\label{ep:20}
\langle \mathcal{M}\phi,\phi \rangle_{\alpha}
=\iint \left( \frac{\phi(x)-\phi(y)}{x-y}\right)^{2}\alpha(dx)\alpha(dy).
\end{equation}
\end{enumerate}
\end{proposition}

\begin{proof} Only the last part of this theorem is not covered in \cite[Proposition 1]{i7}.   
To prove it, we proceed as follows:  Take $\phi$ to be a $C^{2}$ function on $[-2,2]$ and note 
that the variational characterization \eqref{e0:5} gives
\begin{equation}\label{ep:101}
p.v.\int\frac{2}{x-y}\alpha(dy)=x.
\end{equation}
This can then be used to remove the singularity in the definition 
of $\mathcal{M}$ by observing that 
\begin{equation}\label{ep:100}
\begin{split}
(\mathcal{M}\phi)(x)&=2\int\frac{\phi(x)-\phi(y)
-\phi'(x)(x-y)}{(x-y)^{2}}\alpha(dy)+\phi'(x)p.v.\int\frac{2}{x-y}\alpha(dy)\\
&= -2\frac{d}{dx}\int\frac{\phi(x)-\phi(y)}{x-y}\alpha(dy)+\phi'(x)x.
\end{split}
\end{equation}

Next, to show that $\mathcal{M}$ is the counting number 
operator for $\psi_{n}$, notice that, from \eqref{ep:22} and from the 
orthogonality of $\psi_{n}$ with respect to the inner product 
associated with the measure $\alpha$, 
\[
\int \frac{\psi_{n}(x)-\psi_{n}(y)}{x-y}\alpha(dy)=\psi_{n-1}(x),
\]
which, in turn, using \eqref{e0:7b} leads to 
\[
\mathcal{M}\psi_{n}(x)=-2\psi_{n-1}^\prime(x)+x\psi_{n}^\prime(x)=n\psi_{n}(x).  
\]
In other words, $\mathcal{M}$ is the counting number operator for the 
orthonormal basis $\psi_{n}$ of $L^{2}(\alpha)$.  
Finally, to prove \eqref{ep:20}, use the first line of \eqref{ep:100} combined 
with \eqref{ep:101} to justify the following chain of equalities, satisfied 
by any $C^{2}$ function $\phi$ on $[-2,2]$:  
\begin{equation}\label{e0:300}
\begin{split}
\langle \mathcal{M}\phi,\phi\rangle_{\alpha} 
&=2\iint \frac{(\phi(x)-\phi(y))\phi(x)-\phi'(x)\phi(x)(x-y)}{(x-y)^{2}}\alpha(dy) \alpha(dx) 
+ \int x\phi'(x)\phi(x)\alpha(dx)\\   
& = \iint \frac{(\phi(x)-\phi(y))\phi(x)-\phi'(x)\phi(x)(x-y)}{(x-y)^{2}}\alpha(dy) \alpha(dx) \\
& \quad +\iint \frac{(\phi(y)-\phi(x))\phi(y)-\phi'(y)\phi(y)(y-x)}{(x-y)^{2}}\alpha(dx) \alpha(dy) \\
& \quad+ \int x\phi'(x)\phi(x)\alpha(dx) 
\\
& = \iint \frac{(\phi(x)-\phi(y))^{2}-(\phi'(x)\phi(x)-\phi'(y)\phi(y))(x-y)}{(x-y)^{2}}
\alpha(dy) \alpha(dx) + \int x\phi'(x)\phi(x)\alpha(dx) \\
&=\iint \left(\frac{\phi(x)-\phi(y)}{x-y}\right)^{2}\alpha(dx)\alpha(dy) 
- 2\int \phi'(x)\phi(x)\left(p.v.\int\frac{1}{x-y}\alpha(dy)\right) \alpha(dx) \\
& \quad +\int x\phi'(x)\phi(x)\alpha(dx) \\ 
& = \iint \left(\frac{\phi(x)-\phi(y)}{x-y}\right)^{2}\alpha(dx)\alpha(dy). 
\end{split}
\end{equation}
This equality for $C^{2}$ functions can be used in combination with 
standard results of the theory of Dirichlet forms \cite{MR1303354} to justify 
that $\mathcal{M}$ has a unique essentially self-adjoint extension.   
Moreover, standard approximation arguments prove that \eqref{ep:20} is valid for 
any $C^{1}$ function $\phi$.   
\qedhere

\end{proof}

Let us record separately the following important identity.    

\begin{theorem}\label{t:2} For any smooth function $\phi$ on $[-2,2]$, 
\begin{equation}\label{ep:9}
\langle\mathcal{N}\phi,\phi\rangle
=2\langle (\mathcal{M}+I)^{-1}\phi^\prime,\phi^\prime \rangle_{\alpha}.
\end{equation}
\end{theorem}

\begin{proof}  By polarization, \eqref{ep:9} is equivalent to 
\[
\langle\mathcal{N}\phi,\psi\rangle
=2\langle (\mathcal{M}+I)^{-1}\phi^\prime,\psi^\prime \rangle_{\alpha},
\] 
which, by simple approximations, needs only to be verified for $\phi=\phi_{n}$ and 
$\psi=\phi_{m}$.  Now, since $\mathcal{N}\phi_{n}=n\phi_{n}$, \eqref{e0:7} combined 
with $(\mathcal{M}+I)^{-1}\psi_{n-1}=n^{-1}\psi_{n-1}$ and orthogonality, lead to the 
desired conclusion.

\end{proof}

Voiculescu in \cite{V5} introduced the non-commutative derivative  
$\partial:\C[X]\to \C[X]\otimes \C[X]$ which is given by 
\[
\partial 1=0,\quad \partial(X)=1\otimes 1, 
\quad\partial(m_{1}m_{2})=\partial(m_{1})(1\otimes m_{2})+(m_{1}\otimes 1)\partial(m_{2}).  
\]
Particularly useful, is the fact that the non-commutative derivative of 
$P=X^m$ can naturally be identifies as 
\[
\partial P=\frac{P(x)-P(y)}{x-y}.   
\]
For instance, it turns out that \eqref{ep:22} can be nicely rewritten in the form 
\begin{equation}\label{ep:23bis}
\partial \psi_{n}=\sum_{l=0}^{n-1}\psi_{l}\otimes \psi_{n-l-1}.  
\end{equation}

Let us now introduce the natural trace $\alpha$ on $\C[X]$ by 
\[
\alpha(P)=\int Pd\alpha
\] 
and denote by $\alpha^{\otimes k}$ its natural extension to $\C[X]^{\otimes k}$.  
Also introduce the trace $\omega$ on $\C[X]\otimes \C[X]$ through 
\[
\omega(P\otimes Q)=\int P(x)Q(y)\omega(dx ,dy).  
\]  
With these notations, \eqref{ep:21} can be translated into 
\[
\langle \mathcal{N}\phi, \psi \rangle=2\omega(\partial \phi\times \partial\psi).  
\]
In the language of Dirichlet forms this 
simply indicates that $\mathcal{N}$ is the generator of the Dirichlet 
form $\mathcal{D}(P,Q)=2\omega(\partial P\times \partial Q)$.  

In a similar vein, for the operator $\mathcal{M}$,  
\begin{equation}\label{ep:200}
\langle \mathcal{M}\phi, \psi \rangle_{\alpha}
=(\alpha\otimes\alpha)(\partial \phi\times \partial\psi). 
\end{equation}
Next, introducing the dual operator $\partial^{*}$ (see \cite{V5}) via 
\begin{equation}\label{ep:30}
\langle \partial^{*}(\phi\otimes \eta),\psi \rangle_{\alpha}
=\langle\phi\otimes \eta,  \partial \psi \rangle_{\alpha\otimes \alpha},
\end{equation}
a relation which has to be satisfied for $C^{1}$ functions $\phi,\psi,$ and $\eta$,  we see that 
\begin{equation}\label{ep:31}
\mathcal{M}=\partial^{*}\partial, 
\end{equation}
which certainly justifies naming $\mathcal{M}$ the free Ornstein--Uhlenbeck operator.  
Moreover, note that a nice and useful way of defining 
the operator $\partial^{*}$, in terms of the basis $(\psi_n)_{n\geq 0}$, 
is via
\begin{equation}\label{ep:32}
\partial^{*}(\psi_{a}\otimes \psi_{b})=\psi_{a+b+1} \text{ for } a,b\ge0.
\end{equation}

\section{An Interpolation Formula for the Semicircular Law}\label{s:3}

Let us start by recalling a classical interpolation result, e.g., see 
\cite{hps} and the references therein.  
\begin{proposition}\label{p:11}
Let $f,g:\R^{d}\to\R$ be smooth compactly supported functions, then 
\begin{equation}\label{e:1}
\E[f(X)g(X)]-\E[f(X)]\E[g(X)]
=\int_{0}^{1}\E[\langle\nabla f(\sqrt{1-s}X+\sqrt{s} Z),\Sigma \nabla g(\sqrt{1-s}Y+\sqrt{s} Z) \rangle]ds
\end{equation}
where $X,Y,Z$ are $d$-dimensional iid $N(0,\Sigma)$ random vectors.  
\end{proposition}

Replacing $X_{n},Y_{n},Z_{n}$ by iid $GUE(\frac{1}{n})$ ensembles and 
taking $f(A)=g(A)=\Tr_{n}\phi(A)$ the above yields 
\[
\Var(\Tr_{n}\phi(X_{n}))
=\int_{0}^{1}\E\left[\frac{1}{n}\Tr_{n}\left(\phi'(\sqrt{1-s}X_{n}
+\sqrt{s} Z_{n})\phi'(\sqrt{1-s}Y_{n}+\sqrt{s} Z_{n})\right)\right]ds.
\]
Upon taking the limit, as $n\to\infty$, and using fluctuation results for 
random matrices \cite{J,Pastur} or \eqref{ep:23} combined with the general result 
of Voiculescu on freeness \cite{MR1217253}, 
lead to the following formal result.   

\begin{proposition}\label{p:1}  Let $\phi:[-2,2]\to\R$ be a smooth function.   
Then,  
\begin{equation}\label{e1:2}
\iint\left( \frac{\phi(x)-\phi(y)}{x-y}\right)^{2}\omega(dx\,dy)
=\int_{0}^{1}\tau(\phi^\prime(\sqrt{1-s}\mathbf{x}+\sqrt{s} \mathbf{z})\phi^\prime(\sqrt{1-s}\mathbf{y}
+\sqrt{s} \mathbf{z}))ds, 
\end{equation}
where $\mathbf{x},\mathbf{y}$, and $\mathbf{z}$ are free semicircular random variables 
on some non-commutative probability space $(\mathcal{A},\tau)$.   
\end{proposition}

\begin{proof} A proof of \eqref{e1:2} has already been given through random matrix 
manipulations.  However, here is a different and more direct approach:  
From \eqref{ep:21}, the left-hand side of \eqref{e1:2} is:   
\begin{equation}\label{e1:21}
\iint\left( \frac{\phi(x)-\phi(y)}{x-y}\right)^{2}\omega(dx,dy)=\frac{1}{2}
\langle \mathcal{N}\phi,\psi \rangle.  
\end{equation}
To deal with the right-hand side of \eqref{e1:2}, start by observing that 
for a fixed $s\in [0,1]$, the pair $(\sqrt{1-s}\mathbf{x}+\sqrt{s} \mathbf{z},
\sqrt{1-s}\mathbf{y}+\sqrt{s} \mathbf{z})$ is a semicircular system 
as introduced in Section~\ref{s:SS}.  
The covariance matrix of $(\sqrt{1-s}\mathbf{x}+\sqrt{s} \mathbf{z},\sqrt{1-s}\mathbf{y}+\sqrt{s} 
\mathbf{z})$ is easy to compute and is equal to:   
\[
\met{ 1 & s \\ s & 1}.
\]
Since according to Proposition~\ref{p:1SS}, the mixed moments do not depend 
on the particular realization of the semicircular system as long as 
the covariance matrix is the same, 
a different system, namely t$(\sqrt{1-s^{2}}\mathbf{x}+s\mathbf{z},\mathbf{z})$ 
can be chosen.  With this choice,  
\[
\tau(\Phi(\sqrt{1-s}\mathbf{x}+\sqrt{s} \mathbf{z}, \sqrt{1-s}\mathbf{y}+\sqrt{s} \mathbf{z}))
=\tau(\Phi(\sqrt{1-s^{2}}\mathbf{x}+s\mathbf{z},\mathbf{z})),
\]
for any non-commutative polynomial $\Phi$ in two variables.  
In particular, for any smooth 
functions $\phi$ and $\psi$ on $[-2,2]$, it follows that  
\[
\tau(\phi'(\sqrt{1-s}\mathbf{x}+\sqrt{s} \mathbf{z})\psi'(\sqrt{1-s}
\mathbf{y}+\sqrt{s} \mathbf{z})))
=\tau(\phi'(\sqrt{1-s^{2}}\mathbf{x}+s\mathbf{z})\psi'(\mathbf{z})).
\]
From this last fact, combined with the change of variable $s=e^{-t}$, the right-hand side 
of \eqref{e1:2} becomes 
\begin{equation}\label{e1:3}
\int_{0}^{1}\tau(\phi'(\sqrt{1-s^{2}}\mathbf{x}+s\mathbf{z})\psi'(\mathbf{z}))ds
=\int_{0}^{\infty}e^{-t}\tau(\phi'(\sqrt{1-e^{-2t}}\mathbf{x}+e^{-t} 
\mathbf{z})\psi'(\mathbf{z}))dt.  
\end{equation}
Next, define the operator $P_{t}$ via 
\[
\langle P_{t}\phi,\psi \rangle_{\alpha}=\tau(\phi(\sqrt{1-e^{-2t}}\mathbf{x}+e^{-t} 
\mathbf{z})\psi(\mathbf{z})).
\]
From the covariance structure of semicircular systems pointed above, it is easy to 
check that $(P_{t})_{t\ge 0}$ form a semigroup of bounded selfadjoint operators.  
Denote by $-\mathcal{A}$ its generator, which we now plan to identify.  
To this end, take $\phi(x)=x^{a}$ and $\psi(x)=x^{b}$, with $a,b\ge0$ integers, and 
compute  
\[
\langle\mathcal{A}x^{a},x^{b} \rangle
=\frac{d}{dt}\bigg|_{t=0}\tau\left((\mathbf{z}
+\sqrt{2t}\mathbf{x}-t\mathbf{z})^{a}\mathbf{z}^{b}\right) 
= 2\sum_{l_{1}+l_{2}+l_{3}=a-2}\tau(\mathbf{z}^{l_{1}}
\mathbf{x}\mathbf{z}^{l_{2}}\mathbf{x}\mathbf{z}^{l_{3}}\mathbf{z}^{b})
-a\tau(\mathbf{z}^{a}\mathbf{z}^{b}).  
\]
Using the freeness of $\mathbf{x}$  and $\mathbf{z}$, continue with  
\[
2\sum_{l_{1}+l_{2}+l_{3}
=a-2}\tau(\mathbf{z}^{l_{2}})\tau(\mathbf{z}^{l_{1}+l_{3}+b})
-a\tau(\mathbf{z}^{a}\mathbf{z}^{b})=2\sum_{l=0}^{a-2}(l+1)\tau(\mathbf{z}^{l+b})
\tau(\mathbf{z}^{a-2-l})-a\tau(\mathbf{z}^{a+b}),
\]
and thus, since $\mathbf{z}$ is semicircular under $\tau$, arrive at
\begin{equation}\label{e1:5}
-\mathcal{A}\phi=2D(\mathrm{I}\otimes\alpha)(\partial \phi)-x\phi',
\end{equation}
where the operator $D$ is the derivative operator.   

Taking this last identity on functions $\phi=\psi_{n}$, 
combined with \eqref{ep:23bis} and the fact 
that the sequence $\{\psi_{n}\}_{n\ge0}$ is orthogonal with respect to inner product 
associated to $\alpha$, as well as \eqref{e0:7b}  
lead to $\mathcal{A}\psi_{n}=n\psi_{n}$, which shows that $\mathcal{A}=\mathcal{M}$.   
Hence, for smooth functions $\phi$ on $[-2,2]$, the right-hand side 
of \eqref{e1:2} can now be written as
\[\begin{split}
\int_{0}^{\infty}e^{-t}\tau(\phi'(\sqrt{1-e^{-2t}}\mathbf{x}+e^{-t} 
\mathbf{z})\phi'(\mathbf{z}))dt
&=\int_{0}^{\infty}e^{-t}\langle e^{-t\mathcal{M}}\phi',\phi' \rangle_{\alpha}dt\\
&=\int_{0}^{\infty}\langle e^{-t(\mathcal{M}+I)}\phi',\phi'\rangle_{\alpha} dt 
= \langle (\mathcal{M}+I)^{-1}\phi',\phi' \rangle_{\alpha}.
\end{split}
\]

To conclude, the left-hand side of \eqref{e1:2} is 
$\frac{1}{2}\langle\mathcal{N}\phi,\phi \rangle$ 
while its right-hand side is $\langle (\mathcal{M}+I)^{-1}\phi',\phi' \rangle_{\alpha}$, 
and therefore the remaining of the statement follows from Theorem~\ref{t:2}.     

Note that $(P_{t})_{t\ge 0}$ is nothing but the free Ornstein-Uhlenbeck semigroup 
first introduced in \cite{MR1851716}.  \qedhere
\end{proof}

We can now state the following consequence:     

\begin{corollary}[The Free Poincar\'e Inequality]\label{c:5}  
For any smooth function $\phi:[-2,2]\to\R$,  
\begin{equation}\label{efp}
 \int_{-2}^{2}\int_{-2}^{2}\left( \frac{\phi(x)-\phi(y)}{x-y}\right)^{2}
 \frac{(4-xy)}{4\pi^{2}\sqrt{(4-x^{2})(4-y^{2})}}dx\,dy\le \int_{-2}^{2} \phi^\prime(x)^{2}\,\alpha(dx).  
\end{equation}
Equality is only attained for linear functions $\phi$.   
\end{corollary}

\begin{proof}  From the Cauchy-Schwarz inequality, 
\[
\tau(\phi'(\sqrt{1-s}\mathbf{x}+\sqrt{s} \mathbf{z})\phi'(\sqrt{1-s}\mathbf{y}
+\sqrt{s} \mathbf{z}))\le (\tau(\phi^\prime(\sqrt{1-s}\mathbf{x}
+\sqrt{s} \mathbf{z})^{2})^{1/2}\tau((\phi^\prime(\sqrt{1-s}\mathbf{y}
+\sqrt{s} \mathbf{z}))^{2})^{1/2}.
\]
Both, $\sqrt{1-s}\mathbf{x}+\sqrt{s} \mathbf{z}$ 
and $\sqrt{1-s}\mathbf{y}+\sqrt{s} \mathbf{z}$ are semicircular elements 
of variance $1$, thus, for any $s\in [0,1]$, 
\[
\tau(\phi'(\sqrt{1-s}\mathbf{x}+s\mathbf{z})\phi'(\sqrt{1-s}\mathbf{y}
+s \mathbf{z}))\le \int\phi'(x)^{2}\alpha(dx),
\]
which combined with \eqref{e1:2} finishes the proof.

A different proof of \eqref{efp}, follows directly from \eqref{ep:21}, 
Theorem~\ref{t:2} and the fact that $\mathcal{M}$ is non-negative.  
Equality is easily seen to be attained when $\phi'$ is in the kernel of 
$\mathcal{M}$, meaning that $\phi'$ is a constant function, in other words, 
$\phi$ must be a linear.    
\end{proof}

\section{First Refinement}\label{S:1dim}

We now wish to extend the operator $\mathcal{M}$ to tensors.  
To do so, let $\C[X]^{\otimes (k)}:=\C[X]\otimes \C[X]\otimes \dots \otimes \C[X]$ 
where the tensor product is taken $k$ times.   
The non-commutative derivative $\partial:\C[X]\to \C[X]\otimes \C[X]$,  
as it appears in \cite{V5}, is given by
\[
\partial 1=0,\quad \partial(X)=1\otimes 1, \quad\partial(m_{1}m_{2})
=\partial(m_{1})(1\otimes m_{2})+(m_{1}\otimes 1)\partial(m_{2}).  
\]
On monomials $X^a$, $a\ge 1$, this becomes  
\[
\partial(X^{a})=\sum_{p+q=a-1}X^{p}\otimes X^{q}
\]
which is clearly equal to zero for $a=0$.  
The higher derivatives $\partial^{(k)}:\C[X]\to \C[X]^{\otimes (k)}$ are defined 
inductively by $\partial^{(k)}=(\partial\otimes I^{k-1})\partial^{(k-1)}$, and 
it is easy to check that for $0\le k\le n$ and $0\le p$,  
\begin{equation}\label{e:3:1}
\partial^{(p+n)}=(I^{\otimes k}\otimes \partial^{(p)}\otimes I^{\otimes (n-k)})\partial^{(n)}.
\end{equation}

Now, extend the operator $\mathcal{M}$ on $\C[X]$ to 
an operator $\mathcal{M}^{(k)}$ on $\C[X]^{\otimes (k)}$ via:   
\[
\mathcal{M}^{(k)}(P_{1}\otimes P_{2}\dots \otimes P_{k})=(\mathcal{M}P_{1})\otimes P_{2}\dots 
\otimes P_{k}+P_{1}\otimes (\mathcal{M}P_{2})\dots 
\otimes P_{k}+P_{1}\otimes P_{2}\dots \otimes (\mathcal{M}P_{k}). 
\]
Equivalently, this is characterized by 
\begin{equation}\label{e:3:3}
\mathcal{M}^{(a+b)}=\mathcal{M}^{(a)}\otimes I^{\otimes(b)}+I^{\otimes(a)}\otimes \mathcal{M}^{(b)}, 
\end{equation}
for any $a,b\ge0$, with also $\mathcal{M}^{(0)}=I$ and $\mathcal{M}^{(1)}=\mathcal{M}$.    

The following are important properties verified by the operators defined thus far.  

\begin{proposition}\label{p:2}  For any $k\ge1$,  
\begin{equation}\label{e:3:2}
(\partial\otimes I^{\otimes (k-1)})\mathcal{M}^{(k)}=(\mathcal{M}^{(k+1)}+I)
(\partial\otimes I^{\otimes (k-1)})
\end{equation}
while for any polynomials $\phi,\psi\in \C[X]$,
\begin{equation}\label{e:3:30}
\langle \mathcal{M}^{(k)}\partial^{(k-1)}\phi,\partial^{(k-1)}\psi\rangle_{\alpha^{\otimes(k)}}
=k\langle \partial^{(k)}\phi,\partial^{(k)}\psi \rangle_{\alpha^{\otimes (k+1)}}.
\end{equation}
In particular, $\mathcal{M}^{(k)}\partial^{(k-1)}\phi=0$ if and only if 
$\phi$ is a polynomial of degree $k-1$.  

In addition to verifying these properties, the operator $\mathcal{M}^{(k)}$ is essentially 
self-adjoint and non-negative on $L^{2}(\alpha^{\otimes k})$, and for any $a>k-1$, 
\begin{equation}\label{e:3:100}
(\mathcal{M}^{(k)}+a I)^{-1}\partial^{(k-1)} = \partial^{(k-1)}(\mathcal{M}+(1+a-k)I)^{-1}.
\end{equation}
\end{proposition}

\begin{proof}  The proof of \eqref{e:3:2} is done by induction.  
For $k=1$, we need to prove that 
\[
\partial \mathcal{M}=(\mathcal{M}^{(2)}+I)\partial,
\]
and this is going to be verified on polynomials $\psi_{l}$.  
Since 
\begin{equation}\label{tag}
\partial \psi_{l}=\sum_{a+b=l-1}\psi_{a}\otimes \psi_{b}, 
\end{equation}
and since $\mathcal{M}$ is the counting number operator,  
\[
\partial (\mathcal{M}\psi_{l})=l\sum_{a+b=l-1}\psi_{a}\otimes \psi_{b}, \text{ while }\hskip 0.2cm  
\mathcal{M}^{(2)}\partial\psi_{l}=(l-1)\sum_{a+b=l-1}\psi_{a}\otimes \psi_{b},
\]
which is exactly the case $k=1$.  

Now assume $k\ge2$ and use \eqref{e:3:3} to write 
\begin{equation}\label{tagtag}
(\partial\otimes I^{\otimes (k-1)})\mathcal{M}^{(k)}=(\partial\otimes I^{\otimes (k-1)}) 
(\mathcal{M}^{(k-1)}\otimes I)+(\partial\otimes I^{\otimes (k-1)}) 
(I^{\otimes (k-1)}\otimes\mathcal{M}).
\end{equation}
Now, the induction step justifies that 
\[
\begin{split}
(\partial\otimes I^{\otimes (k-1)})(\mathcal{M}^{(k-1)}\otimes I)
&=((\partial\otimes I^{\otimes (k-2)})\mathcal{M}^{(k-1)})\otimes I 
=((\mathcal{M}^{(k)}+I)(\partial\otimes I^{\otimes (k-2)}))\otimes I \\
&=(\mathcal{M}^{(k)}\otimes I)(\partial\otimes I^{\otimes (k-1)})
+\partial\otimes I^{\otimes (k-1)},
\end{split}
\]
while the last term of \eqref{tagtag} is 
\[
(\partial\otimes I^{\otimes (k-1)})(I^{\otimes (k-1)}\otimes\mathcal{M}) 
= (I^{\otimes (k-1)}\otimes\mathcal{M})(\partial\otimes I^{\otimes (k-1)}), 
\]
completing the induction step for \eqref{e:3:2}.  

For $k=1$, the equality \eqref{e:3:30} readily follows 
from \eqref{tag} and since $\mathcal{M}$ is the counting number operator.  
Alternatively, this is just the same as  \eqref{ep:200}.  

For $k\ge2$, using \eqref{e:3:2}, one easily shows that 
\begin{equation}\label{e:3:201}
\mathcal{M}^{(k)}\partial^{(k-1)}=\partial^{(k-1)}(\mathcal{M}-k+1).
\end{equation}
The rest of the proof reduces to showing that  
\[
\langle \partial^{(k-1)}(\mathcal{M}-(k-1)I)\phi,\partial^{(k-1)}
\psi\rangle_{\alpha^{\otimes(k)}}=k\langle \partial^{(k)}\phi,\partial^{(k)}\psi 
\rangle_{\alpha^{\otimes (k+1)}}.
\]
To do so, it is enough to 
take $\phi=\psi_{l}$ and $\psi=\psi_{l'}$ 
for some $l,l'\ge k-1$, i.e., to show that  
\[
(l-k+1)\langle \partial^{(k-1)}\psi_{l},\partial^{(k-1)}\psi_{l'}\rangle_{\alpha^{\otimes(k)}}
=k\langle \partial^{(k)}\psi_{l},\partial^{(k)}\psi_{l'} \rangle_{\alpha^{\otimes (k+1)}}.
\]
Now, an elementary calculation based on \eqref{ep:23bis} reveals that 
\begin{equation}\label{e:3:50} 
\partial^{(k)}\psi_{l}
=\sum_{a_{1}+a_{2}+\dots+a_{k+1}=l-k}
\psi_{a_{1}}\otimes \psi_{a_{2}}\otimes\dots\otimes \psi_{a_{k+1}}, 
\end{equation}
where the summation is over all possible writings of $l-k=a_{1}+a_{2}+\dots +a_{k+1}$, 
with all $a_{1},a_{2},\dots,a_{k+1}\ge0$.  
It remains to show that 
\[
(l-k+1)N_{k-1,l-k+1}\delta_{l,l'}=kN_{k,l-k}\delta_{l,l'}, 
\]
where $N_{k,l}$ is the number of writings of $l=a_{1}+a_{2}+\dots +a_{k+1}$,  
with all $a_{1},a_{2},\dots,a_{k+1}\ge0$.  This follows from the fact that 
\begin{equation}\label{e:3:40}
N_{k,l}= \binom{l+k}{l}, 
\end{equation}
which is well known and easy to verify.  
The last part, namely \eqref{e:3:100}, is obtained in a straightforward fashion from 
\eqref{e:3:201} combined with the fact that $\mathcal{M}$ is a 
self-adjoint non-negative operator.   
\end{proof}

Using the above proposition, a refinement 
of the Poincar\'e inequality 
for the semicircular law can now be stated formally.

\begin{theorem}\label{t:10} For $k\ge1$, and any smooth function $\phi$ on $[-2,2]$, 
\begin{equation}\label{e:t:10}
\begin{split}
\frac{1}{2}\langle\mathcal{N}\phi,\phi \rangle = &\|\phi^\prime \|^{2}_{\alpha}- 
\frac{1}{2}\|\partial\phi^\prime\|^{2}_{\alpha^{\otimes 2}}
+\frac{1}{3}\|\partial^{(2)}\phi^\prime\|^{2}_{\alpha^{\otimes 3}} 
+\dots +\frac{(-1)^{k-1}}{k}\|\partial^{(k-1)}\phi^\prime\|^{2}_{\alpha^{\otimes k}}\\
& \quad +\frac{(-1)^{k}}{k}\langle \mathcal{M}^{(k)}(\mathcal{M}^{(k)}+kI)^{-1}\partial^{(k-1)}
\phi^\prime,\partial^{(k-1)}\phi^\prime \rangle_{\alpha^{\otimes k}}.
\end{split}
\end{equation}
Moreover, $\phi$ is a polynomial of degree $k$ if and only if 
\[
\frac{1}{2}\langle\mathcal{N}\phi,\phi \rangle = \|\phi^\prime \|^{2}_{\alpha}- 
\frac{1}{2}\|\partial\phi'\|^{2}_{\alpha^{\otimes 2}}
+\frac{1}{3}\|\partial^{(2)}\phi^\prime\|^{2}_{\alpha^{\otimes 3}} 
+\dots +\frac{(-1)^{k-1}}{k}\|\partial^{(k-1)}\phi^\prime\|^{2}_{\alpha^{\otimes k}}.  
\]
\end{theorem}

\begin{proof}
We prove \eqref{e:t:10} by induction starting with \eqref{ep:9}.   
For simplicity of notation, denote $\phi'$ by $\psi$ and note that 
\[
 \langle(I+\mathcal{M})^{-1}\psi,\psi \rangle_{\alpha}
 =\langle \psi,\psi \rangle_{\alpha}-\langle \mathcal{M}(\mathcal{M}+I)^{-1}\psi,\psi \rangle_{\alpha}.
\]
Now, using \eqref{e:3:2}, 
$\partial(\mathcal{M}+I)=(\mathcal{M}^{(2)}+2I)\partial$, 
which then leads to 
\[
\langle \mathcal{M}(\mathcal{M}+I)^{-1}\psi,\psi \rangle_{\alpha}
=\langle \partial(\mathcal{M}+I)^{-1}\psi,\partial\psi \rangle_{\alpha^{\otimes 2}}
=\langle (\mathcal{M}^{(2)}+2I)^{-1}\partial\psi,\partial\psi \rangle_{\alpha^{\otimes 2}}.
\]

Moreover, by repeating this argument, the last term above becomes:  
\[
\langle (\mathcal{M}^{(2)}+2I)^{-1}\partial\psi,\partial\psi \rangle_{\alpha^{\otimes 2}}
=\frac{1}{2}\langle \partial\psi,\partial\psi \rangle_{\alpha^{\otimes 2}}
-\frac{1}{2}\langle \mathcal{M}^{(2)}(\mathcal{M}^{(2)}+2I)^{-1}\partial\psi,\partial\psi 
\rangle_{\alpha^{\otimes 2}}.
\]
Now that we saw the mechanics on how to proceed, we can formally 
prove the inductive step by showing that the formula \eqref{e:t:10} for $k\ge1$, implies 
the case $k+1$.  To do so, using \eqref{e:3:201}, \eqref{e:3:30} and 
again \eqref{e:3:201}, allow to first justify that 
\[
\begin{split}
\langle \mathcal{M}^{(k)}&(\mathcal{M}^{(k)}+kI)^{-1}\partial^{(k-1)}
\psi,\partial^{(k-1)}\psi \rangle_{\alpha^{\otimes k}}
=\langle \mathcal{M}^{(k)}\partial^{(k-1)}(\mathcal{M}+I)^{-1} \psi,\partial^{(k-1)}\psi 
\rangle_{\alpha^{\otimes k}} \\ 
&=k\langle \partial^{(k)}(\mathcal{M}+I)^{-1} \psi,\partial^{(k)}\psi \rangle_{\alpha^{\otimes k}} \\ 
&=k\langle(\mathcal{M}^{(k+1)}+(k+1)I)^{-1}\partial^{(k)} \psi,\partial^{(k)}\psi
\rangle_{\alpha^{\otimes k}} \\
&=\frac{k}{k+1}\langle \partial^{(k)} \psi,\partial^{(k)}\psi
\rangle_{\alpha^{\otimes k}} -\frac{k}{k+1}\langle(\mathcal{M}^{(k)}(\mathcal{M}^{(k+1)}+(k+1)I)^{-1}\partial^{(k)} \psi,\partial^{(k)}\psi\rangle_{\alpha^{\otimes k}}.
\end{split}
\]
Therefore,
\[
\begin{split}
&\frac{1}{k}\langle \mathcal{M}^{(k)}(\mathcal{M}^{(k)}+kI)^{-1}\partial^{(k-1)}
\psi,\partial^{(k-1)}\psi \rangle_{\alpha^{\otimes k}} \\
&\qquad =\frac{1}{k+1}\| \partial^{(k)} \psi\|_{\alpha^{\otimes k}}^{2}
-\frac{1}{k+1}\langle(\mathcal{M}^{(k)}
(\mathcal{M}^{(k+1)}+(k+1)I)^{-1}\partial^{(k)} \psi,\partial^{(k)}\psi\rangle_{\alpha^{\otimes k}},
\end{split}
\]
proving the main induction step.  

It is also clear that the last term in \eqref{e:t:10} is zero since for a 
polynomial $\phi$ of degree $k$, $\partial^{(k-1)}\phi'$ is constant, 
and $\mathcal{M}^{(k)}$ vanishes on constants.   
\end{proof}

As a consequence of Theorem~\ref{t:10}, we also have the following result. 

\begin{corollary}  For any $k\ge1$ and any smooth function $\phi$ on $[-2,2]$, 
\[
\sum_{l=1}^{2k}\frac{(-1)^{l-1}}{l}\|\partial^{(l-1)}\phi^\prime\|^{2}_{\alpha^{\otimes l}}
\le \frac{1}{2}\langle\mathcal{N}\phi,\phi 
\rangle\le 
\sum_{l=1}^{2k-1}\frac{(-1)^{l-1}}{l}\|\partial^{(l-1)}\phi^\prime\|^{2}_{\alpha^{\otimes l}}.
\]
Above, equality is attained on the left-hand side for any polynomial $\phi$ of degree $2k$, while 
on the right-hand side it is attained for any polynomial $\phi$ of degree $2k-1$.   

Also, 
\[
\frac{1}{2}\langle\mathcal{N}\phi,\phi \rangle 
= \sum_{l=1}^{\infty}\frac{(-1)^{l-1}}{l}\|\partial^{l-1}\phi'\|^{2}_{\alpha^{\otimes l}},
\]
provided the series converges (for instance,  
this is always the case if $\phi$ is a polynomial).   
\end{corollary}

\section{Heuristics}\label{s:heur}

The main purpose of the present section is to give heuristic 
arguments justifying the presence of the operators involved in the proof 
of the free Brascamp-Lieb inequality.

We start with the classical case.   On $\R^{d}$, consider a probability measure of the 
form $\mu(dx)=e^{-V(x)}dx$, where $V$ is a smooth function on $\R^{d}$.  
The measure $\mu_{V}$ is the invariant measure of the operator 
$L=-\Delta+\nabla V\cdot \nabla$ which in turn is a generalization of the 
Ornstein-Uhlenbeck operator.  A simple integration by parts leads to  
\begin{equation}\label{e2:1}
\E_{\mu}[ (Lf) g ]= \E_{\mu}[\langle \nabla f,\nabla g\rangle]. 
\end{equation}

One of the classical approaches to the Brascamp-Lieb inequality 
is due to Helffer \cite{MR1624506} and we quickly review it here.  
First, from  \eqref{e2:1} with $f$ replaced by $Lf$, 
\begin{equation}\label{e3:1:0}
\langle L\phi,L\phi \rangle_{L^{2}(\mu)}
=\langle \nabla L\phi,\nabla \phi\rangle_{L^{2}(\mu)}.
\end{equation}

Second, with the natural (component-wise) extension of $L$ to several dimensions, 
\[
\nabla L =L\nabla+Hess V\nabla=(L+Hess V)\nabla.  
\]
In particular, if $K:=L+Hess V$, then 
\begin{equation}\label{e:c}
\nabla L=K\nabla,
\end{equation}
and so, using inverses whenever these are defined, 
\begin{equation}\label{e3:1:1}
K^{-1}\nabla=\nabla L^{-1}. 
\end{equation}

If $f$ is a smooth function such that 
$\int f\,d\mu=0$, \eqref{e3:1:0} with 
$\phi=L^{-1}f$ and \eqref{e3:1:1} lead to 
\begin{equation}\label{e3:1:4}
\Var_{\mu}(f)=\langle K^{-1}\nabla f, \nabla f \rangle_{L^{2}(\mu)}.
\end{equation}
Since, $L$ is a non-negative operator and 
since $Hess V$ is positive definite, it follows that 
$Hess V\le K$ and so $K^{-1}\le (Hess V)^{-1}$, from which the Brascamp-Lieb inequality
\begin{equation}\label{e3:1:5}
\Var_{\mu}(f)\le\langle (Hess V)^{-1}\nabla f, \nabla f \rangle_{L^{2}(\mu)}
\end{equation}
follows naturally.    

We now wish to apply these types of arguments to the case 
of $\mu=\p_{V}^{n}$.  To do so, let $\phi:\R\to\R$ be non-constant,  
compactly supported and smooth, and let $f(X)=\Tr \phi(X)$, 
for any $X\in\mathcal{H}_{n}$.   For a better understanding, we actually back 
up a step  
and start with 
\begin{equation}\label{e3:1:10}
\Var_{\p_{V}^{n}}(f)=\langle \nabla ((L^{n}_{V})^{-1}\Tr \phi), 
\nabla \Tr\phi \rangle_{\p_{V}^{n}}
=\langle (K^{n}_{V})^{-1}\nabla\Tr \phi, \nabla \Tr\phi \rangle_{\p_{V}^{n}}.   
\end{equation}

We next wish to understand what happens if we let $n$ tend to infinity in \eqref{e3:1:4}.  
The limit of the left-hand side is determined by the fluctuations of random matrices, and 
(say, provided that $V$ is a polynomial of even degree with equilibrium 
measure $\mu_{V}$ having support $[-2,2]$), it is given by:    
\[
\lim_{n\to\infty}\Var_{\p_{V}^{n}}(f)=\int_{-2}^{2}\int_{-2}^{2}
\left( \frac{\phi(x)-\phi(y)}{x-y}\right)^{2}
\frac{(4-xy)}{4\pi^{2}\sqrt{(4-x^{2})(4-y^{2})}}dx\,dy.
\]

For the right-hand side of \eqref{e3:1:10} observe first 
that $\nabla \Tr \phi(X)=\phi^\prime(X)$.  Now, 
 \[
-L^{n}_{V}\Tr \phi(X)=\Delta\Tr \phi(X)-n\nabla \Tr V(X)\cdot \nabla \Tr \phi(X)
=\Delta\Tr 
\phi(X)-n\Tr(V^\prime(X)\phi^\prime(X)),
\]
hence we need to identify the limit of the operators $K^{n}_{V}$.   
 
Hence, the Laplacian on matrices needs to be computed, and 
we do so for monomials of the form $\phi(X)=X^{a}$ and then 
extend the result by linearity.  By definition, 
\[
\Delta \Phi(X)=\sum_{\gamma} \frac{d^{2}}{d h^{2}}\Phi(X+h E_{\gamma}),
\] 
where $E_{\gamma}$ is an orthonormal basis of $\mathcal{H}_{n}$.  
In fact, a basis consists of the matrices $E_{jj}$ which have 
$1$ on the $j$th position on the diagonal and $0$ elsewhere, $E_{jk}$ with $j<k$ 
with $1/\sqrt{2}$ for the $(j,k)$th and $(k,j)$th entries and $0$ otherwise, 
and $\tilde{E}_{jk}$ with $i/\sqrt{2}$ for the $(j,k)$th entry and $-i/\sqrt{2}$ for 
the $(k,j)$th entry and $0$ otherwise.  
Using this basis,  
\[
\begin{split}
\frac{1}{2}(\Delta \Tr \phi)(X)
&=\sum_{1\le j\le k\le n} \sum_{l_{1}+l_{2}+l_{3}=a-2}\Tr( X^{l_{1}} E_{jk}X^{l_{2}}E_{jk}X^{l_{3}})
+\sum_{1\le j<k\le n} \sum_{l_{1}+l_{2}+l_{3}=a-2}\Tr( X^{l_{1}} 
\tilde{E}_{jk}X^{l_{2}}\tilde{E}_{jk}X^{l_{3}})\\
&=\sum_{1\le j\le k\le n} \sum_{l_{1}+l_{2}+l_{3}=a-2}\Tr( X^{l_{1}+l_{3}} E_{jk}X^{l_{2}}E_{jk})
+\sum_{1\le j<k\le n} \sum_{l_{1}+l_{2}+l_{3}=a-2}\Tr( X^{l_{1}+l_{3}} 
\tilde{E}_{jk}X^{l_{2}}\tilde{E}_{jk}) \\
&=\sum_{l=0}^{a-2}(l+1)\sum_{1\le j\le k\le n} \Tr( X^{l} E_{jk}X^{a-2-l}E_{jk})
+\sum_{l=0}^{a-2}(l+1)\sum_{1\le j<k\le n} \Tr( X^{l} \tilde{E}_{jk}X^{a-2-l}\tilde{E}_{jk}).
\end{split}
\] 
Let $F_{jk}$ be the matrix with $1$ on the $(j,k)$th entry and $0$ otherwise.  
Then $E_{jk}=(F_{jk}+F_{kj})/\sqrt{2}$ and $\tilde{E}_{jk}=i(F_{jk}-F_{kj})/\sqrt{2}$, for $j<k$.  
A small computation reveals that for two matrices $A$ and $B$, 
\begin{equation}\label{e3:1:12}
\begin{split}
\sum_{1\le j\le k\le n} \Tr( A E_{jk}BE_{jk}) +\sum_{1<j<k\le n}\Tr(A \tilde{E}_{jk} B\tilde{E}_{jk})
&=\sum_{j,k=1}^{n} \Tr( A F_{jk}BF_{jk})) \\ 
& =\sum_{j,k=1}^{n} \sum_{u_{1},u_{2},u_{3},u_{4}=1}^{n}A_{u_{1}u_{2}} (F_{jk})_{u_{2}u_{3}}B_{u_{3}u_{4}}(F_{kj})_{u_{4}u_{1}} \\
&=\sum_{j,k=1}^{n} A_{jj}B_{kk}=\Tr(A)\Tr(B).
\end{split}
\end{equation}
Therefore, 
\[
(\Delta \Tr\phi)(X)=2\sum_{l=0}^{a-2}(l+1)\Tr(X^{l})\Tr(X^{a-2-l}).  
\] 
Summing up our findings, for $\phi(x)=x^{a}$, $a\ge2$, 
\[
\begin{split}
-L^{n}_{V}\Tr \phi(X)&=2\sum_{l=0}^{a-2}(l+1)\Tr( X^{l})\Tr(X^{a-2-l})-na\Tr(V'(X)X^{a-1}) \\
&=\sum_{l=0}^{a-2}(l+1)\Tr( X^{l})\Tr(X^{a-2-l})
+\sum_{l=0}^{a-2}(a-1-l)\Tr( X^{l})\Tr(X^{a-2-l})-na\Tr(V'(X)X^{a-1}) \\
&=\Tr\otimes\Tr(\partial \phi'(X))-n\Tr(V'(X)\phi'(X)),
\end{split}
\]
which, by linearity, is then true for all polynomials $\phi$.  
Taking the gradient, then gives
\[
\begin{split}
-\nabla (L^{n}_{V}\Tr\phi)(X)&=\nabla \Tr\otimes \Tr(\partial 
\phi^\prime(X))-n(V^\prime\phi^\prime)^\prime
=(D\otimes \Tr)\partial \phi^\prime(X)+(\Tr\otimes D) \partial 
\phi^\prime(X)-n(V^\prime\phi^\prime)^\prime(X).
\end{split}
\]
In particular, using the operator $K_{V}^{n}$ which satisfies  
$\nabla (L^{n}_{V}\Tr\phi)(X)=K^{n}_{V}\nabla \Tr \phi(X)$, we now obtain 
\[
\frac{1}{n}K_{V}^{n}\nabla \Tr\phi (X)=-\frac{1}{n}(D\otimes \Tr 
+ \Tr \otimes D)\partial \phi^\prime(X)+(V^\prime\phi^\prime)^\prime(X), 
\]
and therefore, since $\partial \phi$ is a symmetric tensor,  
\[
\frac{1}{n}K_{V}^{n}\phi^\prime(X)
=-\frac{2}{n}D(I\otimes \Tr)\partial \phi^\prime(X)+(V^\prime\phi^\prime)^\prime(X).
\]
Finally, since $\Tr \psi(X)/n$ converges to $\mu_{V}(\psi)$, heuristically  
$K^{n}_{V}\phi^\prime/n$ converges to 
\[
K\phi^\prime=-2D(I\otimes \mu_{V})\partial \phi^\prime+(V^\prime\phi^\prime)^\prime
\]
and replacing $\phi^\prime$ by $\phi$, motivates the following definition:   
\begin{equation}\label{e3:4}
\mathcal{K}_{V}\phi:=-D[2(I\otimes \mu_{V})\partial \phi-V^\prime\phi]
=-D[2(I\otimes \mu_{V})\partial \phi]+V^\prime\phi^\prime +V^{\prime\prime} \phi.
\end{equation}
It is interesting to remark that using, for instance, 
\cite[Corollary 4.4 and Proposition 3.5]{V5} 
one can justify the following equality
\[
-D[2(I\otimes \mu_{V})\partial \phi]+V^\prime\phi^\prime=\partial^{*}_{V}\partial,
\]
where $\partial^{*}_{V}$ is the adjoint of the operator $\partial$, i.e., 
$\langle \partial^{*}_{V}(\phi\otimes\psi),\eta \rangle_{\mu_{V}}=\langle 
\phi\otimes \psi, \partial \eta \rangle_{\mu_{V}\otimes\mu_{V}}$ and the inner 
product generated by a state $\mu$ on polynomials $\C\langle X\rangle$ 
is $\langle \phi,\psi \rangle_{\mu}=\mu(\phi\,\bar{\psi})$, for any polynomials 
$\phi,\psi$ with the convention that $\bar{\psi}=\sum_{i}\bar{a}_{i}X^{i}$ 
if $\psi=\sum_{i}a_{i}X^{i}$.   The extension to tensor products is done via  
the usual procedure: 
$$\langle \phi_{1}\otimes \phi_{2}, \psi_{1}\otimes\psi_{2}  \rangle_{\mu\otimes\mu}=\mu(\phi_{1}\bar{\psi}_{1})\mu(\phi_{2}\bar{\psi}_{2}),$$   
and the representation therefore 
obtained has the flavor of a generalization of the non-commutative 
Ornstein-Uhlenbeck operator.   

Finally, heuristically, taking the limit in \eqref{e3:1:10}, it follows that 
\begin{equation}
\frac{1}{2}\langle \mathcal{N}f,f \rangle
=\langle \mathcal{K}_{V}^{-1}f^\prime,f^\prime \rangle_{L^{2}(\mu_{V})}.
\end{equation}

\section{The Free Brascamp-Lieb Inequality}\label{s:7}

From the heuristics of the previous section, let  
\[
\mathcal{K}_{V}\phi:=D\left[-2(I\otimes \mu_{V})\partial \phi +V^\prime\phi \right], 
\]
and set 
\[
\mathcal{M}_{V}\phi:=-2D(I\otimes \mu_{V})\partial \phi+V^\prime\phi^\prime.
\]
Again, if $V(x)=x^{2}/2$, then $\mathcal{M}_{V}$ is the counting number operator 
for the Chebyshev polynomials of the 
second kind.   It is clear that, 
\[
\mathcal{K}_{V}\phi=\mathcal{M}_{V}\phi+V^{\prime\prime}\phi,  
\]
and it is trivial that, for any function $f$, the multiplication operator 
\[
\mathcal{A}_{f}\phi=f\phi, 
\]
extends to a self-adjoint operator on $L^{2}(\mu_{V})$.  
As shown next, the operator $\mathcal{M}_{V}$ is a non-negative operator 
on $L^{2}(\mu_{V})$.

\begin{proposition}\label{p:100}  Assume $\mu_{V}$ has support $[-2,2]$.  
The operator $\mathcal{M}_{V}$ is given on $C^{2}$ functions by
\begin{equation}\label{e4:1}
(\mathcal{M}_{V}\phi)(x)=2p.v.\int \frac{\phi(x)-\phi(y)}{(x-y)^{2}}\mu_{V}(dy) 
:= 2\lim_{\epsilon\searrow0}\int_{|x-y|\ge\epsilon} \frac{\phi(x)-\phi(y)}{(x-y)^{2}}\mu_{V}(dy),
\end{equation}
and for $C^{1}$ functions, $\phi,\psi$ on $[-2,2]$, 
\begin{equation}\label{e4:2}
\langle \mathcal{M}_{V}\phi,\psi \rangle_{V}
=\int \frac{(\phi(x)-\phi(y))(\psi(x)-\psi(y))}{(x-y)^{2}}\mu_{V}(dx)\mu_{V}(dy).
\end{equation}
Moreover, $\mathcal{M}_{V}$ can be extended to an unbounded non-negative essentially 
self-adjoint operator on $L^{2}(\mu_{V})$ whose domain includes 
the set of $C^1$ functions on $[-2,2]$.  

In addition, if $V''\ge0$, and $V$ is $C^{2}$ on $[-2,2]$, then the operator $\mathcal{K}_{V}$ 
has a self-adjoint extension such that for some $\delta>0$, $\mathcal{K}_{V}\ge \delta I$.  
In particular, $\mathcal{K}_{V}$ is an essentially self-adjoint operator on $L^{2}(\mu_{V})$ 
which is invertible with a bounded inverse on $L^{2}(\mu_{V})$.   
\end{proposition}

\begin{proof}
The statements in the first part of this proposition, namely, \eqref{e4:1} and \eqref{e4:2} follow from 
arguments similar to those involved in the proof of \eqref{ep:20} from Proposition~\ref{p:10}.  
Start with a $C^{2}$ function $\phi$ on $[-2,2]$ and notice that 
\[
\partial\phi=\frac{\phi(x)-\phi(y)}{x-y},   
\]
and thus 
\[
(I\otimes\mu_{V})(\partial \phi)(x)=\int \frac{\phi(x)-\phi(y)}{x-y}\mu_{V}(dy).
\]
Next, $\phi$ is a $C^{2}$ function which when combined with  
the variational characterization of the equilibrium measure from \eqref{e0:5}, leads to  
\begin{equation}\label{e4:6}
\begin{split}
-2\frac{d}{dx}\int \frac{\phi(x)-\phi(y)}{x-y}\mu_{V}(dy)
&=2\int \frac{\phi(x)-\phi(y)-\phi'(x)(x-y)}{(x-y)^{2}}\mu_{V}(dy) \\ 
&= 2p.v.\int \frac{\phi(x)-\phi(y)}{(x-y)^{2}}\mu_{V}(dy) 
-\phi^\prime(x)p.v.\int\frac{2}{x-y}\mu_{V}(dy)\\ 
&=2p.v.\int \frac{\phi(x)-\phi(y)}{(x-y)^{2}}\mu_{V}(dy)
-V^\prime(x)\phi^\prime(x),
\end{split}
\end{equation}
giving \eqref{e4:1}.   In turn, \eqref{e4:2} follows easily for $C^{2}$ functions exactly 
as in \eqref{e0:300}, replacing $\alpha$ by $\mu_{V}$.  Next, the extension to 
a self-adjoint operator is deduced from the fact that the Dirichlet form 
\[
\mathcal{D}(\phi,\phi)=\int\left( \frac{\phi(x)-\phi(y)}{x-y}\right)^2\mu_{V}(dx)\mu_{V}(dy)
\]
is positive and closable, therefore  $\mathcal{M}_V$, its generator, according 
to \cite{MR1303354} must be essentially self-adjoint and non-negative.  
This last fact and standard 
approximations of $C^{1}$ functions with $C^{2}$ functions proves 
\eqref{e4:2} for $C^{1}$ functions.   

Since $V$ is a $C^{2}$ function, the multiplicative operator 
$\mathcal{A}_{V''}$ is a bounded 
operator on $L^{2}(\mu_{V})$ and this implies, for instance, that 
$\mathcal{K}_{V}$ has a 
non-negative extension with the same domain as $\mathcal{M}_{V}$.   
In addition, we claim that $V''>0$ on a set of positive measure (with respect to $\mu_{V}$).   
Indeed if otherwise, then $V''$ would be identically 0 on $[-2,2]$ which means, for 
example, that $V'$ is constant on $[-2,2]$.  Since the support 
of $\mu_{V}$ is $[-2,2]$, it follows that 
(e.g., see \cite[Theorem 1.11 Chapter IV]{ST} or \cite[Equation (4.3)]{i7}) 
\begin{equation}\label{e4:21}
\int V'(x)\beta(dx)=0 \text{ and } \int xV'(x)\beta(dx)=2.
\end{equation}
These equalities  cannot be satisfied if $V'$ is constant on $[-2,2]$.   
Therefore, there must be a subset $A\subset [-2,2]$ with $\mu_{V}(A)>0$ and a 
positive $\epsilon>0$ such that $V''(x)>\epsilon$, for all $x\in A$.  
On the other hand, for $x,y\in[-2,2]$, $\frac{1}{(x-y)^{2}}\ge \frac{1}{16}$, 
and then 
\[
\begin{split}
\langle \mathcal{K}_{V}\phi,\phi \rangle
&\ge \frac{1}{16}\int(\phi(x)-\phi(y))^{2}\mu_{V}(dx)\mu_{V}(dy)+\epsilon \int_{A}\phi^{2}d\mu_{V}\\ 
&= \epsilon \int_{A}\phi^{2}d\mu_{V}+\frac{1}{8}\int\phi^{2}d\mu_{V}-\frac{1}{8}\left(\int \phi\,d\mu_{V}\right)^{2}.
\end{split}
\]
Next, we wish to show that there exists $\delta>0$ such that 
\[
\epsilon \int_{A}\phi^{2}d\mu_{V}
+\frac{1}{8}\int\phi^{2}d\mu_{V}-\frac{1}{8}\left(\int \phi\,d\mu_{V}\right)^{2}
\ge \delta \int\phi^{2}d\mu_{V},
\]
or equivalently,
\[
(1+8\epsilon-8\delta) \int_{A}\phi^{2}d\mu_{V}+(1-8\delta)\int_{A^{c}}\phi^{2}d\mu_{V}
\ge \left(\int \phi\,d\mu_{V}\right)^{2}.
\]
To show this is possible, notice that if $A$ has full measure, then we can 
take $\delta=\epsilon$ and we are done.   If not, then $\mu_{V}(A)>0$ and 
$\mu_{V}(A^{c})>0$.   
By the Cauchy-Schwarz inequality,  
\[
(1+8\epsilon-8\delta) \int_{A}\phi^{2}d\mu_{V}+(1-8\delta)\int_{A^{c}}\phi^{2}d\mu_{V}\ge \frac{(1+8\epsilon-8\delta)}{\mu_{V}(A)} \left(\int_{A}\phi d\mu_{V}\right)^{2}
+\frac{1-8\delta}{\mu_{V}(A^{c})}\left(\int_{A^{c}}\phi d\mu_{V}\right)^{2},
\]
and then $(a^{2}+b^{2})(c^{2}+d^{2})\ge (ac+bd)^{2}$, with 
$a=\sqrt{\frac{1+8\epsilon-8\delta}{\mu_{V}(A)}}\int_{A}\phi d\mu_{V}$, 
$b=\sqrt{\frac{1-8\delta}{\mu_{V}(A^{c})}}\int_{A^{c}}\phi d\mu_{V}$, $c=\sqrt{\frac{\mu_{V}(A)}{1+8\epsilon-8\delta}}$, and $d=\sqrt{\frac{\mu_{V}(A^{c})}{1-8\delta}}$ yields 
\[
(1+8\epsilon-8\delta) \int_{A}\phi^{2}d\mu_{V}+(1-8\delta)\int_{A^{c}}\phi^{2}d\mu_{V}
\ge\frac{1}{\frac{\mu_{V}(A)}{1+8\epsilon-8\delta}
+\frac{\mu_{V}(A^{c})}{1-8\delta}}\left(\int \phi d\mu_{V}\right)^{2}.
\]
Thus, we just need to choose $\delta>0$ such that 
$$\frac{\mu_{V}(A)}{1+8\epsilon-8\delta}+\frac{\mu_{V}(A^{c})}{1-8\delta}<1,$$  
which is certainly possible since the above quantity is continuous in $\delta$, and 
since for $\delta=0$, 
$$\frac{\mu_{V}(A)}{1+8\epsilon}+\mu_{V}(A^{c})<\mu_{V}(A)+\mu_{V}(A^{c})=1.$$   
The rest now follows.   
 \qedhere
\end{proof}

\begin{theorem}\label{t:11}
Let the support of the equilibrium measure $\mu_{V}$ be $[-2,2]$ and let 
$V$ be $C^{4}$ with $V''\ge0$ on $[-2,2]$.   Then, for any $C^{1}$ function 
$\phi$ on $[-2,2]$, 
\begin{equation}\label{e4:5}
\langle \mathcal{N}\phi,\phi \rangle=2\langle \mathcal{K}^{-1}_{V}\phi^\prime,
\phi^\prime \rangle_{L^{2}(\mu_{V})}.
\end{equation}
Moreover, the following version of Brascamp-Lieb inequality holds true:   
For any $C^{1}$ function $\phi$ on $[-2,2]$,   
\begin{equation}\label{e4:3}
\int_{-2}^{2}\int_{-2}^{2}\left( \frac{\phi(x)-\phi(y)}{x-y}\right)^{2}
\frac{(4-xy)}{4\pi^{2}\sqrt{(4-x^{2})(4-y^{2})}}dx\,dy
\le \int \frac{\phi'^{2}}{V^{\prime\prime}}d\mu_{V}, 
\end{equation}
with equality for $\phi(x)=V^\prime(x)+C$, $C\in \R$.   
\end{theorem}

We use the $C^{4}$ regularity of $V$ at a single place in the proof, and so 
most likely this assumption can be reduced to $C^{3}$ regularity, but we are not pursuing 
this here.   

\begin{proof}
We want to show that the right-hand side of \eqref{e4:5} is in fact independent of $V$, 
so it suffices to check it for a potential which is smooth, convex and whose 
equilibrium measure is also supported on $[-2,2]$.  
That candidate is precisely $V(x)=x^{2}/2$ and then the rest of the statement 
is just Theorem~\ref{t:2}.  

Let the operator $\mathcal{F}_{V}$  be defined as 
\[
(\mathcal{F}_{V}\phi)(x)=2\int \log|x-y|\phi(y)\mu_{V}(dy). 
\]
From \eqref{ep:40}, we know that 
$d\mu_{V}=(1-\frac{1}{2}\mathcal{N}V)d\beta$, hence for 
any $C^{2}$ function $\phi$ on $[-2,2]$, 
\[
(\mathcal{F}_{V}\phi)(x)=2\int \log|x-y|\phi(y)(1-\frac{1}{2}\mathcal{N}V(y))\beta(dy)
=-2\mathcal{E}\left(\left(1-\frac{1}{2}\mathcal{N}V\right)\phi\right)(x)
=-2\left(\mathcal{E}\mathcal{A}_{(1-\frac{1}{2}\mathcal{N}V)}\phi\right)(x).
\]
Now, taking the derivative yields, 
\[
\begin{split}
(D\mathcal{F}_{V}\phi)(x)&=2p.v.\int\frac{\phi(y)}{x-y}\mu_{V}(dy)
=-2\int\frac{\phi(x)-\phi(y)}{x-y}\mu_{V}(dy)+2p.v.\int\frac{\phi(x)}{x-y}\mu_{V}(dy) \\ 
&=-2\int\frac{\phi(x)-\phi(y)}{x-y}\mu_{V}(dy)+\phi(x)V^\prime(x).
\end{split}
\]
Taking another derivative and using \eqref{e4:6} give, 
\[
\begin{split}
(D^{2}\mathcal{F}_{V}\phi)(x)
& =(\mathcal{M}_{V}\phi)(x)-V^\prime(x)\phi^\prime(x)+(\phi V^\prime)^\prime(x)\\
&=(\mathcal{M}_{V}\phi)(x)+V^{\prime\prime}(x)\phi(x)=(\mathcal{K}_{V}\phi)(x).
\end{split}
\]
Thus, $\mathcal{K}_{V}=-2D^{2}\mathcal{E}\mathcal{A}_{(1-\frac{1}{2}\mathcal{N}V)}$,  
on $C^{2}([-2,2])$.  To finish the proof we want to take inverses.  
To do so requires to properly define the inverses and to this end, we look at 
the following diagram: 
\begin{equation}\label{e4:20}
C^{2}([-2,2])\xrightarrow[]{\mathcal{A}_{(1-\frac{1}{2}\mathcal{N}V)}} C^{2}([-2,2])\xrightarrow[]{\mathcal{E}} C^{2}([-2,2]) \cap L^{2}_{0}(\beta)\xrightarrow[]{D} C^{1}([-2,2]) \xrightarrow[]{D} C([-2,2]).
\end{equation} 
We need to justify that the composition is well defined here.  In the first place, and 
for instance, from \eqref{e4:21}, and the very definition of the operator $\mathcal{N}$, 
it follows that 
\begin{equation}\label{e4:22}
1-\frac{1}{2}(\mathcal{N}V)(x)=\frac{4-x^{2}}{2}\int\frac{V'(x)-V'(y)}{x-y}\beta(dy).  
\end{equation}
This proves two things.   In the first place, since $V$ is convex, we learn that 
$u(x):=\int\frac{V'(x)-V'(y)}{x-y}\beta(dy)$, is strictly positive and 
also $C^{2}$ on $[-2,2]$ (this is the only place where the $C^{4}$ condition on $V$ is used).   
Thus, the first operator is well defined.  The second operator is also 
well defined by Proposition~\ref{p:10}, while the other operators are self-explanatory.   
The inverses are written as follows:  
\begin{equation}\label{e4:20bis}
C([-2,2])\xrightarrow[]{\mathcal{I}} C^{1}([-2,2]) \xrightarrow[]{\mathcal{I}_{0}} C^{2}([-2,2]) \cap L^{2}_{0}(\beta)\xrightarrow[]{\mathcal{N}} C([-2,2]) \xrightarrow[]{\mathcal{A}_{1/(1-\frac{1}{2}\mathcal{N}V)}} L^{2}(\nu_{V}),
\end{equation} 
where 
\[
(\mathcal{I}\psi)(x)=\int_{-2}^{x}\psi(y)dy 
\text{ and } (\mathcal{I}_{0}\psi)(x)=(\mathcal{I}\psi)(x)-\int \mathcal{I}\psi\,d\beta,
\]
with $\nu_{V}(dx)=(4-x^{2})\mu_{V}(dx)$.  Clearly, by definition,  the 
operator $\mathcal{N}$ maps $C^{2}([-2,2])$ into $C([-2,2])$, 
while $\mathcal{A}_{1/(1-\frac{1}{2}\mathcal{N}V)}$ 
sends $C([-2,2])$ into $L^{2}(\nu_{V})$, because of \eqref{e4:22}.  
The natural choice here would be to have the 
operator $\mathcal{A}_{1/(1-\frac{1}{2}\mathcal{N}V)}$ 
map $C([-2,2])$ into $L^{2}(\mu_{V})$,  
but this works only for functions which vanish 
(like a power grater than $1/2$ of $x$) at the endpoints of $[-2,2]$.  
Therefore, instead of restricting the domain of definitions of all 
the other operators to accommodate this fact, we change the range 
where $\mathcal{A}_{1/(1-\frac{1}{2}\mathcal{N}V)}$ takes values.   
On the other hand, the operator $\mathcal{K}_{V}^{-1}$ is 
a bounded operator from $L^{2}(\mu_{V})$ into itself 
by Proposition~\ref{p:100}, hence it can be taken as a bounded operator 
from $L^{2}(\mu_{V})$ into $L^{2}(\nu_{V})$.     

Thus, we can now argue that on the set of continuous functions on $[-2,2]$, 
$\mathcal{K}_{V}^{-1}=-\frac{1}{2}\mathcal{A}_{1/(1-\frac{1}{2}\mathcal{N}V)}
\mathcal{N}\mathcal{I}_{0}\mathcal{I}$. 
This, combined with  
$\langle \phi,\psi \rangle_{\mu_{V}}=\langle \mathcal{A}_{(1-\frac{1}{2}\mathcal{N}V)}\phi,\psi \rangle$,  
now results with 
\[
\langle \mathcal{K}^{-1}_{V}\phi^\prime,\phi^\prime \rangle_{L^{2}(\mu_{V})}
=-\frac{1}{2}\langle\mathcal{A}_{1-\frac{1}{2}\mathcal{N}V}
\mathcal{A}_{1/(1-\frac{1}{2}\mathcal{N}V)}\mathcal{N}\mathcal{I}_{0}\mathcal{I} D\phi,D\phi \rangle
=-\frac{1}{2}\langle\mathcal{N}\mathcal{I}_{0}\mathcal{I} D\phi,D\phi \rangle, 
\]
which is independent of $V$!   The rest of \eqref{e4:5} follows as we 
pointed out by taking $V(x)=x^{2}/2$ and using Theorem~\eqref{t:2}.  

To verify \eqref{e4:3}, notice that it suffices to show that for 
any $\phi\in L^{2}(\mu_{V})$,
\[
\langle (\mathcal{M}_{V}+\mathcal{A}_{V''})^{-1}\phi,\phi \rangle_{\mu_{V}}
\le \left\langle \frac{1}{V''}\phi,\phi \right\rangle_{\mu_{V}}.
\]
Since $\mathcal{K}_{V}$ is invertible, 
any $\phi\in L^{2}(\mu_{V})$ can be written as $\phi=(\mathcal{M}_{V}+V'')\psi$, 
for some $\psi\in L^{2}(\mu_{V})$, so we need to check that 
\begin{equation}\label{e:4:500}
\langle (\mathcal{M}_{V}+V'')\psi,\psi \rangle_{\mu_{V}}\le \langle (V'')^{-1}(\mathcal{M}_{V}+V'')\psi,(\mathcal{M}_{V}+V'')\psi \rangle_{\mu_{V}}. 
\end{equation}
If the right-hand side is infinite, there is nothing to prove.  
If it is finite, then write it as 
\[
\begin{split}
&\langle (V'')^{-1}(\mathcal{M}_{V}+V'')\psi,(\mathcal{M}_{V}+V'')\psi \rangle_{\mu_{V}} \\
& \qquad \qquad = \langle (\mathcal{M}_{V}+V'')\psi,\psi \rangle_{\mu_{V}}
+\langle (V'')^{-1/2}\mathcal{M}_{V}\psi,(V'')^{-1/2}\mathcal{M}_{V}\psi \rangle_{\mu_{V}}
+\langle \mathcal{M}_{V}\psi,\psi \rangle_{\mu_{V}},
\end{split}
\]
from which \eqref{e:4:500} follows immediately.   
There is, however, a small detail 
we need to take care of, namely justifying that if the left-hand side of the 
above is finite, the equality above is well defined.   
This essentially boils down to showing  that all terms on the right-hand side are finite.   
This is indeed so because the finiteness of the left-hand side is 
equivalent to $(V'')^{-1/2}(\mathcal{M}_{V}+V'')\psi \in L^{2}(\mu_{V})$ which, in 
particular, since $V''$ is continuous and $\psi\in L^{2}(\mu_{V})$, is equivalent 
to $(V'')^{-1/2}\mathcal{M}_{V}\psi\in L^{2}(\mu_{V})$.  
This is sufficient to guarantee the validity of the equation above, ensuring in particular 
that the middle term on the right-hand side is finite.     

For the case of equality, according to \eqref{ep:21} we have to show that 
\begin{equation}\label{e4:10}
\langle \mathcal{N}V^\prime,V^\prime \rangle = 2\int V^{\prime\prime}d\mu_{V}.
\end{equation}
To do so, use \cite[Eq (1.32)]{i7} which gives 
\[
\langle \mathcal{N}\phi,\psi^\prime \rangle+\langle \mathcal{N}\psi,\phi^\prime \rangle
=\left(\int\phi^\prime d\beta\right)\left(\int x\psi'(x) \beta(dx)\right)
+\left(\int x\phi'(dx) \beta(dx)\right)\left(\int\psi^\prime d\beta\right).
\]
Taking $\phi=V^\prime$ and $\psi=V$, results with 
\[
\langle \mathcal{N}V^\prime,V^\prime \rangle+\langle \mathcal{N}V,V^{\prime\prime} \rangle
=\left(\int V^{\prime\prime}d\beta\right)\left(\int x V'(x)\beta(dx)\right)+\left(\int x 
V''(x)\beta(dx)\right)\left(\int V^{\prime}d\beta\right).
\]
On the other hand, since the equilibrium measure $\mu_{V}$ is supported 
on $[-2,2]$, invoking the equation \eqref{e4:21} yields the equality 
\[
\langle \mathcal{N}V^\prime,V^\prime \rangle+\langle \mathcal{N}V,V^{\prime\prime} \rangle
=2\int V^{\prime\prime}d\beta,
\]
or written differently,  
\[
\langle \mathcal{N}V^\prime,V^\prime \rangle = 2\int V^{\prime\prime} 
\left(1-\frac{1}{2}\mathcal{N}V\right)d\beta,
\]
which combined with \eqref{ep:40} is precisely the statement of \eqref{e4:10}.   
\qedhere

\end{proof}

The curious reader may wonder how the free Brascamp-Lieb looks like in the case where 
the support of the measure $\mu_{V}$ is not $[-2,2]$.  
Assuming that the equilibrium measure $\mu_{V}$ has support $[a,b]$ and 
that $V$ is $C^{4}$ on $[a,b]$ with  $V''(x)\ge 0$ for $x\in [a,b]$, the analog of 
inequality \eqref{e4:3} takes the form
\begin{equation}\label{e:400}
\int_{a}^{b}\int_{a}^{b}\left(\frac{\phi(x)-\phi(y)}{x-y}\right)^{2}
\frac{-2ab+(a+b)(x+y)-2xy}{8\pi^{2}\sqrt{(x-a)(b-x)}\sqrt{(y-a)(b-y)} } \, dxdy
\le \int \frac{\phi'^{2}}{V''}d\mu_{V},
\end{equation}
for any smooth function $\phi$ on $[a,b]$.  
The proof of this is simply done by a linear rescaling, namely reducing everything to the 
case $[a,b]=[-2,2]$.  More precisely, take $\theta(x)=(b-a)x/4+(b+a)/2$ which maps 
$[-2,2]$ into $[a,b]$.   
With this, \eqref{e:400} reduces to \eqref{e4:3} for 
$\tilde{V}(x)=V(\theta(x))$, $\tilde{\phi}(x)=\phi(\theta(x))$.  
Notice here that the equilibrium measure $\mu_{\tilde{V}}$ is determined 
by $\mu_{\tilde{V}}(A)=\mu_{V}(\{ x\in[a,b]:\theta^{-1}(x)\in A\})$, for any $A\subset[-2,2]$.
  
Equality in \eqref{e:400} is attained for functions $\phi$ of the form  
$\phi(x)=c_{1}+c_{2}V'(x)$, for some constants $c_{1},c_{2}$.

We close this section with an extension of \cite[Eq. 10.16]{i5} which seems mysterious 
there, but is demystified by the free Brascam-Lieb inequality discussed here.   
This inequality is related to the Wishart random matrix models and the main potential $V$ 
is defined only on the positive axis.  The interested reader can take a look at \cite{i5} 
for more details.      

\begin{corollary} Let $Q:[0,\infty)\to\R$ be a continuous function 
and let $V(x)=Q(x)-s\log(x)$, for $s>0$, be such that $\lim_{x\to\infty}(V(x)-2\log (x))=\infty$.   
Let the support of $\mu_{V}$ be $[a,b]$ and on $[a,b]$, let $Q$ be $C^{4}$ and such that $Q''\ge0$.  
Then, for any smooth function $\phi$ on $[a,b]$,  
\begin{equation}\label{e:203}
\int_{a}^{b}\int_{a}^{b}\left(\frac{\phi(x)-\phi(y)}{x-y}\right)^{2}
\frac{-2ab+(a+b)(x+y)-2xy}{8\pi^{2}\sqrt{(x-a)(b-x)}\sqrt{(y-a)(b-y)}} \, dxdy
\le  \int \frac{x^{2}\phi'(x)^{2}}{s+x^{2}Q''(x)}\mu_{V}(dx),
\end{equation}
with equality for $\phi(x)=c_{1}(Q'(x)-s/x)+c_{2}$, for some constants $c_1$ and $c_2$.  

In particular we obtain \cite[Eq. 10.16]{i5}
\begin{equation}\label{e:204}
s\int_{a}^{b}\int_{a}^{b}\left(\frac{\phi(x)-\phi(y)}{x-y}\right)^{2}
\frac{-2ab+(a+b)(x+y)-2xy}{8\pi^{2}\sqrt{(x-a)(b-x)}\sqrt{(y-a)(b-y)}} \, dxdy
\le  \int x^{2}\phi'(x)^{2}\mu_{V}(dx).
\end{equation}
 
If $Q(x)=rx+t$, for some constants $r$ and $t$, 
\eqref{e:204} is sharp with equality attained 
for $\phi(x)=c_{1}+{c_{2}/x}$.  

\end{corollary}

\begin{proof}  From  \eqref{e:400} and since  
$V''(x)=Q''(x)+s/x^{2}\ge s/x^{2}$ one immediately deduces 
\eqref{e:203}.  Equality in \eqref{e:204} is attained if $Q''(x)=0$ and $\phi(x)=c_{1}+c_{2}/x$.

\end{proof}

\section{Final Remarks}  It is clearly of interest to discuss a multidimensional 
version of the 
free Poincar\'e inequality and extensions, as for instance in the 
spirit of \cite{i10}.  
This requires more work and it will eventually be done in a separate publication.   

There is a version of the free Poincar\'e inequality, introduced by Biane in \cite{Biane2},  
and in the one dimensional case it is different from the one presented here.   
It is interesting to point out that in several dimensions, the fluctuations of 
jointly independent random matrices, more precisely the limiting variance of the fluctuations, 
are the main ingredients for the formulation of the free Poincar\'e inequality.   
This already appears in the literature in two different forms.  
One is in \cite{MR2216446}, which describes it in terms of second order freeness.  
The other is investigated in \cite{MR2353386}, and the variance term is 
given in a form similar to the one presented in \eqref{e4:5}.   

There are however some noticeable differences between the one dimensional case and 
the multidimensional case.   If we interpret the variance term in the 
Poincar\'e inequality described in \eqref{e4:5} as 
$\langle\mathcal{K}_{V}^{-1}\phi',\phi'\rangle_{\mu_{V}}$, then the key statement is 
that, as long as the support of the measure $\mu_{V}$ is $[-2,2]$, this variance does not 
depend on the other details of the potential $V$.  
This is, in some sense, reminiscent of the universality of fluctuations in 
random matrix theory.   As it turns out, this fact does not seem to take place in 
several dimensions which means that the approach for proving 
Theorem~\ref{t:11} is not going to work.  

To fully understand the multidimensional case it seems desirable to 
unify the two points of view mentioned above, namely, the second order 
freeness and the analog of the variance through the inverse of a properly defined operator, 
at least for some natural cases of potentials.   

\section*{Acknowledgments}  We would like to thank the anonymous referees for pointing 
out typos, errors, ambiguities and complementing this with very useful suggestions 
which improved this paper considerably.

\bibliographystyle{amsplain}

\end{document}